\documentclass[11pt]{amsart}

\usepackage{amscd}
\usepackage{amsmath, amssymb, comment}
\usepackage{amsfonts}
\usepackage{enumerate}
\usepackage{color}

\usepackage{hyperref}
\usepackage{url}
\newcommand{\de}{\partial}

\newcommand{\ov}[1]{\overline{#1}}

\newcommand{\ddbar}{\sqrt{-1} \partial \overline{\partial}}
\newcommand{\Ric}{\mathrm{Ric}}

\newcommand{\vp}{\varphi}
\newcommand{\vol}{\mathrm{vol}}
\newcommand{\diam}{\mathrm{diam}}
\newcommand{\ve}{\varepsilon}
\newcommand{\e}{\varepsilon}

\newcommand{\p}{\partial}

\renewcommand{\e}{\varepsilon}

\newcommand{\un}[1]{\underline{#1}}

\newcommand{\dist}{\mathrm{dist}}

\newcommand{\ul}[1]{\underline{#1}}

\newcommand{\abs}[1]{\left\lvert#1\right\rvert}
\newcommand{\norm}[1]{\left\lVert#1\right\rVert}

\renewcommand{\leq}{\leqslant}
\renewcommand{\geq}{\geqslant}
\renewcommand{\le}{\leqslant}

\newcommand{\be}{\begin{equation}}
\newcommand{\ee}{\end{equation}}

\newcommand{\mSH}{m\mathrm{SH}}
\newcommand{\SH}{\mathrm{SH}}

\newcommand{\R}{\mathbb{R}}
\newcommand{\C}{\mathbb{C}}

\newcommand{\Hm}{\mathrm{H}^m}

\begin{document}
\newcounter{remark}
\newcounter{theor}
\setcounter{theor}{1}
\newtheorem{claim}{Claim}
\newtheorem{theorem}{Theorem}[section]
\newtheorem{lemma}[theorem]{Lemma}
\newtheorem{corollary}[theorem]{Corollary}
\newtheorem{proposition}[theorem]{Proposition}
\newtheorem{question}{Question}[section]
\newtheorem{goal}{Goal}[section]
\newtheorem{definition}[theorem]{Definition}
\newtheorem{remark}[theorem]{Remark}

\numberwithin{equation}{section}

\title[Eigenvalue Problem]{The eigenvalue problem for the complex hessian operator on $m$-pseudoconvex manifolds}

\author[J. Chu]{Jianchun Chu}
\address{School of Mathematical Sciences, Peking University, Yiheyuan Road 5, Beijing 100871, People's Republic of China}
\email{jianchunchu@math.pku.edu.cn}

\author[Y. Liu]{Yaxiong Liu}
\address{Department of Mathematics,
University of Maryland,
4176 Campus Dr,
College Park, MD 20742}
\email{yxliu238@umd.edu}

\author[N. McCleerey]{Nicholas McCleerey}
\address{Department of Mathematics,
Purdue University, West Lafayette
150 N University St
West Lafayette, IN 47907}
\email{nmccleer@purdue.edu}

\begin{abstract}
We establish $C^{1,1}$-regularity and uniqueness of the first eigenfunction of the complex Hessian operator on strongly $m$-pseudoconvex manifolds, along with a variational formula for the first eigenvalue. From these results, we derive a number of applications, including a bifurcation-type theorem and geometric bounds for the eigenvalue.
\end{abstract}


\maketitle

\section{Introduction}

\subsection{Main Results} 
A classical observation of Lions \cite{Lions85} is that the real Monge-Amp\`ere operator admits a well-defined first eigenvalue on a strictly convex domain with smooth boundary; somewhat surprisingly, this eigenvalue shares many properties with the first eigenvalue of a general linear elliptic operator, despite the fact that the Monge-Amp\`ere operator is fully non-linear. Since then, eigenvalues of non-linear operators have generated lots of interest and been studied by many authors (see Subsection \ref{Previous Works} below).

In this work, we will be concerned with the eigenvalue problem for certain complex operators, namely the $m$-Hessian operators, building off of a pair of recent papers by Badiane-Zeriahi  \cite{BZ23a, BZ23b}. Let us briefly recall some of their results.

Let $\Omega\subset\C^n$ be a smoothly bounded, sufficiently pseudoconvex domain, $\omega$ the Euclidean metric on $\C^n$, $m$ an integer $1\leq m\leq n$, and $0\leq f \in L^p(\Omega)$ for some $p > n$. Then Badiane-Zeriahi consider the problem of finding a number $\lambda_1 > 0$ and a function $u_1\in \mSH(\Omega)\cap C^2(\Omega)$ solving:
\begin{equation}\label{CHE eigenvalue}
\begin{cases}
(\ddbar u_1)^m\wedge \omega^{n-m} = (-\lambda_1 u_1)^m f^m\omega^n & \text{ in }\Omega,\\
u_1 = 0 & \text{ on }\p\Omega,\\
\inf_\Omega u_1 = -1.
\end{cases}
\end{equation}
The number $\lambda_1 := \lambda_1(\Omega ,f)$ is the first (twisted) eigenvalue of the complex $m$-Hessian operator $\Hm(u):= (\ddbar u)^m\wedge\omega^{n-m}/\omega^{n}$, and $u_1$ is the corresponding (normalized) eigenfunction. As is well-known, when $m = n$, $\Hm = \mathrm{MA}$, the complex Monge-Amp\`ere operator, and when $m = 1$, $\Hm = \tfrac{1}{n}\Delta_\omega$, the $\omega$-Laplacian operator.

When $m = n$ and $f > 0$ is smooth up to the boundary, Badiane-Zeriahi \cite[Theorem 1.1]{BZ23a} show the existence of a unique pair $(\lambda_1, u_1)$ solving \eqref{CHE eigenvalue}; they moreover show that $u_1$ is smooth on the interior of $\Omega$ and $C^{1,\alpha}$ up to the boundary, for all $0 < \alpha < 1$. This corresponds to the $C^{1,1}$ boundary regularity shown in Lions for the eigenfunction of the real Monge-Amp\`ere operator (both follow from {\it a priori} Laplacian bounds). Badiane-Zeriahi also prove a Rayleigh quotient formula for $\lambda_1$, in terms of the Monge-Amp\`ere energy, similar to results of Tso \cite{Tso90} and Wang \cite{Wang94} for the real Monge-Amp\`ere and $m$-Hessian equations (although the techniques in \cite{BZ23a, BZ23b} are quite different from those in \cite{Tso90, Wang94}).

When $\frac{n-1}{2} < m \leq n$, Badiane-Zeriahi again show the Rayleigh quotient formula for $\lambda_1$ \cite[Theorem 1.1]{BZ23b} (which implies that $\lambda_1$ is unique in this case), and also the existence of a H\"older continuous eigenfunction $u_1$. Their techniques are unable to establish higher regularity or uniqueness in this case (this can be partially explained by noting that their methods apply to much more general right-hand sides, where higher regularity should not hold -- see \cite[Theorem 1.2]{BZ23b}).

Our goal in this paper is to expand upon \cite{BZ23a, BZ23b}, and answer some special cases of questions they pose. Our main result is the following:

\begin{theorem}\label{Theorem 1}

Suppose that $\Omega$ is an strongly $m$-pseudoconvex manifold $($see Section \ref{Background} for a definition$)$, $\omega$ is a K\"ahler metric on $\Omega$, and $0 < f \in C^\infty(\ov{\Omega})$. Then there exists some $\lambda_1 := \lambda_1(\Omega, f) > 0$ and $u_1\in \mSH(\Omega)\cap C^\infty(\Omega)\cap C^{1,1}(\ov{\Omega})$ such that $(\lambda_1, u_1)$ is the unique solution to \eqref{CHE eigenvalue}.

\end{theorem}

Theorem \ref{Theorem 1} answers a question posed in \cite[page 3, line 7]{BZ23a} when $m = n$, and answers important special cases of \cite[Question 1 and 2]{BZ23b} when the right-hand side function is smooth. Note that we do not require any restriction on the value of $m$. Even further, we show that uniqueness holds for weak solutions which, {\it a priori}, are only assumed to be in the Cegrell class $\mathcal{E}^1_m(\Omega)$; see Theorem \ref{uniqueness in E1}. Our proof of uniqueness crucially uses the $C^{1,1}(\ov{\Omega})$ regularity of $u_1$.

Moreover, our results apply to abstract strongly $m$-pseudoconvex manifolds $\Omega$ (see Section \ref{Background} for the precise definition we use). In general, there may be no embedding of $\Omega$ into $\C^n$ (e.g. if $\Omega$ contains a proper closed subvariety of (positive) dimension $\leq n-m$). As such, we also obtain a generalization of Badiane-Zeriahi's Rayleigh quotient formula for $\lambda_1$:

\begin{theorem}\label{Theorem 2}
Suppose we are in the setting of Theorem \ref{Theorem 1}. Then the first eigenvalue of $\Hm$ satisfies:
\[
\lambda_{1}(\Omega ,f)^{m} = \inf\left\{\frac{E_{m}(u)}{I_{m}(u)}~\middle|~u\in\mathcal{E}_{m}^{1}(\Omega),\ u\neq0\right\},
\]
where we define:
\[
E_{m}(u) = \frac{1}{m+1}\int_{\Omega}(-u)(\ddbar u)^{m}\wedge\omega^{n-m},
\]
and
\[
I_{m}(u) = \frac{1}{m+1}\int_{\Omega}(-u)^{m+1} f^m\omega^{n}.
\]

\end{theorem}

From our main results, we deduce several applications. The first is a sufficient condition to solve certain Hessian equations whose right-hand side depends on the unknown function, without having to assume the existence of a subsolution:

\begin{theorem}\label{Theorem 3}
Let $\psi(z,s)$ be a smooth (strictly) positive function on $\ov{\Omega}\times(-\infty,0]$ such that $\p_s \psi \geq -\gamma_0 > -\lambda_1$, where $\lambda_1 = \lambda_1(\Omega, 1)$ is the first eigenvalue of $\Hm$ associated to $\omega^n$.  Then the equation:
\begin{equation}
\begin{cases}
\sigma_m(u) = \psi^m(z, u) &\text{ in }\Omega,\\
u = 0 &\text{ on }\p\Omega
\end{cases}
\end{equation}
admits a unique solution $u\in \mSH(\Omega)\cap C^\infty(\ov{\Omega})$.
\end{theorem}

When $\p_s\psi \geq 0$, this is a well-known result of Caffarelli-Kohn-Nirenberg-Spruck \cite{CKNS85}; Theorem \ref{Theorem 3} weakens this to an essentially optimal condition. Again, the importance of results such as Theorem \ref{Theorem 3} comes from the fact that they do not require the existence of a subsolution -- while the subsolution theory has proven to be highly successful in many applications (particularly to complex geometry), it is often quite difficult to verify the existence of a subsolution directly. As such, it is beneficial to avoid this assumption when possible.

In order to apply Theorem \ref{Theorem 3} effectively however, one needs control over $\lambda_1$. To this end, we present some geometric estimates for $\lambda_1$, including a lower bound for $\Omega\subset \C^n$ (Theorem \ref{lower bound of lambda_1}) and a monotonicity statement (Theorem \ref{monotonicity}) that generalizes \cite[Theorem 5.6]{BZ23a}. Our results are inspired by work of Le \cite{Le18}, for the real Monge-Amp\`ere eigenvalue.

The comparison of the first eigenvalue of Laplacian, as established by Cheng \cite{Ch75}, is the classical result on Riemannian manifolds, which provides an upper bound of the first eigenvalue of Laplacian.
The key element in its proof is the Laplacian comparison of the distance function under certain curvature condition. Using the Laplacian comparison together with Rayleigh quotient formula (Theorem \ref{Theorem 2}), we can derive an upper bound on manifolds with non-negative Ricci curvature:
\begin{theorem}\label{upper bound manifold}
Suppose that the Ricci curvature $\Ric(\omega)$ is non-negative. Let $R > 0$ be the largest number such that there exists a geodesic ball $B_R(p)\subseteq \Omega$ with $\dist_\omega^{2}(z, p)$ smooth and plurisubharmonic on $B_R(p)$. Then:
\[
\lambda_1(\Omega, f)  \leq c(n,m) (\inf_\Omega f)^{-\frac{2m+1}{m+1}}  R^{-\frac{2n}{m+1} - 2} \diam(\Omega)^{\frac{2n}{m+1}}\vol(\Omega)^{-\frac{1}{m+1}}
\|f\|_{L^{m}(\Omega)}^{\frac{m}{m+1}}.
\]
\end{theorem}
\noindent In particular, when $\Omega\subset\C^n$, the number $R$ is the in-radius.

\subsection{Remarks on the Proofs}

We now briefly discuss some of the main points in our proofs. As already mentioned, our paper builds directly on the general outlines in Badiane-Zeriahi \cite{BZ23a, BZ23b}, which in turn build upon work of Lions \cite{Lions85} and Wang \cite{Wang94} for the real Monge-Amp\`ere operator. We use the idea of \cite{CTW19} to control the largest eigenvalue of the real Hessian for solutions, which is what allows us to show that $u_1$ is $C^{1,1}(\ov{\Omega})$; in fact, we prove more general $C^2$-estimates up to the boundary, under the assumption of a subsolution and supersolution (Theorems \ref{existence under sub-super} and \ref{existence under sub-super degenerate}). Our boundary Hessian estimates differ from the previous work of Collins-Picard \cite{CP22} in that we require our constants to be independent of a lower bound for $\psi$, which we accomplish by utilizing the supersolution and $m$-pseudoconvexity of $\Omega$.

Needing to produce a supersolution creates additional difficulties when we attempt to apply our Hessian estimates to the eigenvalue problem; in \cite[Theorem 1.1]{BZ23a}, a crucial step in overcoming these difficulties is to use an explicit {\it a priori} gradient bound for solutions to the complex Monge-Amp\`ere equation, due to B\l ocki \cite{Bl09}. The corresponding estimate for the complex $m$-Hessian equation is a long standing open problem. We overcome this difficulty by using instead the stability estimate of Dinew-Ko\l odziej \cite{DK14}.

Finally, we are able to improve upon the methods in \cite{BZ23b} by employing both a robust integration by parts argument due to Le \cite{Le18}, and a modification of the classical Hopf lemma in our proof of uniqueness (Theorem \ref{uniqueness in E1}). The integration by parts argument only needs Dinew-Lu's mixed Hessian inequality \cite{DL15}, which allows us to avoid having to show any regularity of the minimizer of the Rayleigh quotient. Our Hopf lemma argument uses the $C^{1,1}(\ov{\Omega})$ regularity of $u_1$ to show some extra positivity for the linearization of $\sigma_m^{1/m}(u_1)$ near the boundary of $\Omega$, where the operator is not uniform elliptic.

\subsection{Previous Works}\label{Previous Works}

Since Lions \cite{Lions85}, many authors have studied eigenfunctions for the real Hessian equations; general existence and uniqueness of the eigenfunction for smooth domains was established by Wang \cite{Wang94}. For the real Monge-Amp\`ere equation, recently Savin \cite{Savin14} has established $C^2$ regularity of the eigenfunction up to the boundary. Subsequently, Le-Savin \cite{LS17} were able to extend this to smoothness up to the boundary. The results in \cite{Savin14} are based strongly on convex geometric techniques, which cannot be applied to the complex setting. For general bounded convex domains, Le \cite{Le18} has recently established existence and uniqueness of the eigenfunction and eigenvalue for the real Monge-Amp\`ere operator, along with several interesting applications.

The variational characterization of the eigenvalue presented here is originally due to Tso \cite{Tso90}, who proved the equivalent characterization for the real Hessian equations. He also studied a family of eigenvalue problems by varying the power of $-u$ on the right-hand side. Very recently, generalizations of this family have been studied by Tong-Yau \cite{TY23} and forthcoming work of Collins-Tong-Yau, with applications to the existence of complete Calabi-Yau metrics on complements of divisor.

In the complex setting, comparably less work has been. Aside from the works of Badiane-Zeriahi \cite{BZ23a, BZ23b}, which we've already discussed at length, there are a pair of papers by Koutev-Ramadanov \cite{KR89, KR90}, which discuss some applications of the Monge-Amp\`ere eigenfunction, similar in spirit to Theorem \ref{Theorem 3}. There is also some work done in the case when $\Omega$ is the unit ball in $\C^n$ \cite{KR89, Li16, Wei16, LD21} (unlike the real case, there is apparently no general theory which shows that the eigenfunction is radially symmetric in this setting).

Finally, we mention a number of works on Sobolov-type inequalities for $m$-subharmonic and $m$-convex functions, which have been relevant to some of the previous works listed above. The basic version we need, which is recalled in Section \ref{Background} below, is due to B\l ocki \cite{Bl93}; the real analogues were shown in \cite{Wang94, Tso90}. More powerful Sobolev inequalities for $m$-sh function have recently been shown by Zhou \cite{Zhou13}, \AA hag-Czy\.{z} \cite{AC20}, and Wang-Zhou \cite{WZ22}.

\subsection{Further Directions}

Our work presents many interesting follow-up questions. First, it is a natural question to ask about higher order boundary regularity of the eigenfunction. This is likely very difficult; the proof of boundary regularity for the real Monge-Amp\`ere eigenfunctions \cite{Savin14, LS17} is quite involved, and, as already mentioned, seems to have little to offer the complex setting.

Second, it would be interesting to sharpen several of the results presented. For instance, it seems possible that one might be able to weaken $m$-pseudoconvexity of the boundary. We also believe that it might be possible to improve our geometric bounds for the eigenvalue so that it depends only on the $L^{n}(\Omega)$-norm of $f$.

Another potentially more difficult direction of inquiry would be to investigate if the supersolution condition in Theorem \ref{existence under sub-super} could be removed -- this would follow from an improved normal-normal estimate for the full hessian on the boundary. The difficulty is that the estimate would need to be independent of a lower bound for $\psi$, which is why the method of Collins-Picard \cite{CP22} cannot be applied verbatim.

Finally, it would be interesting to study the Tso families in the complex setting as well \cite{Tso90}; it is known in the real setting that, as the family parameter varies,  the solutions converge to the solution of the limiting parameter. Additionally, the ``sub-critical" case enjoys better variational properties than the ``critical case" (which is the one discussed in this paper). Taken together, this suggests an alternate approach to tackling some of the questions raised at the end of \cite{BZ23b}.

\subsection{Outline}

We conclude this introduction with an outline of the rest of the paper.
In Section \ref{Background} we recall several background results on strongly $m$-pseudoconvex manifolds and the complex $m$-Hessian equation.
In Section \ref{A Priori Estimates} we prove {\it a priori} $C^2$ estimates up to the boundary for certain Dirichlet problems, and then use them to solve the Dirichlet problem for right-hand sides which are decreasing in $u$ (assuming the existence of a subsolution and supersolution) in Section \ref{Existence theorem}; the proof uses an iteration argument. In Section \ref{Existence and Uniqueness of the Eigenfunction}, we prove our main results, Theorems \ref{Theorem 1} and \ref{Theorem 2}. In Section \ref{Applications}, we presents some applications, and prove Theorem \ref{Theorem 3}. Finally, we collect some further results on pluripotential theorem for $m$-subharmonic functions on strongly $m$-pseudoconvex manifolds in the Appendix \ref{Some Pluripotential Theory}.

\medskip

{\bf Acknowledgments:} The authors would like to thank Weijun Zhang for bringing the eigenvalue problem for complex Hessian operators to our attention. We would also like to thank Xingyu Zhu, Laszlo Lempert, Fried Tong, Valentino Tosatti, Bin Zhou, and Jiakun Liu for helpful discussions and suggestions. The first-named author was partially supported by National Key R\&D Program of China 2023YFA1009900, NSFC grant 12271008 and the Fundamental Research Funds for the Central Universities, Peking University.

\section{Background Results on Strongly \texorpdfstring{$m$}{m}-Pseudoconvex Manifolds}
\label{Background}

\subsection{Definitions}

Suppose that $\ov{\Omega}$ is a compact K\"ahler manifold with (non-empty) smooth boundary; we write $\Omega$ for the interior of $\ov{\Omega}$. Fix a K\"ahler metric $\omega$ on $\ov{\Omega}$ (by which we mean that, in local holomorphic coordinates centered at $z\in \p\Omega$, $\omega$ can always be extended as a smooth K\"ahler metric on an open neighborhood of $z$). We write $\mSH(\Omega) := \mSH(\Omega, \omega)$ for the space of $\omega$-$m$-subharmonic functions on $\Omega$; since we regard $\omega$ as fixed, we will often omit it in the notation/terminology.

As in the case of a domain, we write $\mathcal{E}_m^1(\Omega, \omega) \subset \mSH(\Omega)$ (or just $\mathcal{E}_m^1(\Omega)$) for the finite energy class of $m$-sh functions, which will be the largest space we consider in this paper. Recall that functions in $\mathcal{E}^1_m(\Omega)$ have zero  ``boundary values" in a weak sense, which allows for, e.g., integration by parts  -- see Subsection \ref{subsection: energy} and Appendix \ref{Some Pluripotential Theory} for definitions and further details. For $u \in \mathcal{E}^1_m(\Omega)$, the complex $m$-Hessian operator is well-defined. We shall variously write:
\[
\Hm(u) = (\ddbar u)^m\wedge\omega^{n-m}
\]
for this operator, which in general is only a non-negative Radon measure on $\Omega$. When $\Hm(u)$ is absolutely continuous with respect to $\omega^n$, we also define:
\[
\sigma_m(u) :=  \binom{n}{m} \frac{\Hm(u)}{\omega^n}.
\]
This is the case of course when $u \in C^2(\ov{\Omega})$. Then $\sigma_m(u) \in C^0(\ov{\Omega})$, and $\sigma_m(u)$ can moreover be computed by evaluating the $m^\text{th}$ elementary symmetric polynomial on the eigenvalues of $\ddbar u$, with respect to $\omega$.

\medskip

\begin{definition}
We shall say that $\Omega$ is a {\bf strongly $m$-pseudoconvex manifold} if there exists a negative, strictly $m$-sh function $\rho\in \mSH(\Omega)\cap C^\infty(\ov{\Omega})$ such that $\ddbar \rho - \e\omega$ is an $m$-positive form for some $\e > 0$, $\{\rho = 0\} = \p\Omega$, and $d\rho\not= 0$ on $\p\Omega$.
\end{definition}

This definition is a generalization of the \cite[Definition 2.3]{Ngu14}. As usual, we understand $d\rho|_{\p\Omega}$ by looking at local smooth extensions of $\rho$ in coordinate balls.

From a potential theoretic viewpoint, there seems to be only minor differences between strongly $m$-pseudoconvex manifolds and bounded strongly $m$-pseudoconvex domains (which we always mean to be open subsets of $\C^n$). There are however some important geometric differences, unlike in the more familiar plurisubharmonic case -- for instance, strongly $m$-pseudoconvex manifolds can contain non-trivial closed complex submanifolds (see e.g. \cite{CM23}).

The rest of this section is devoted to recalling some background results we will need. Before doing this, we state three definitions which we will use repeatedly throughout the paper.

\begin{definition}\label{definitions}
Suppose that $\psi$ is a smooth function on $\ov{\Omega}\times (-\infty, 0]$, and write $z\in \Omega$. We consider the PDE:
\begin{equation}\label{psi equation}
\begin{cases}
(\ddbar u)^m\wedge \omega^{n-m} = \psi^m(z, u)\omega^n & \text{ in }\Omega,\\
u = 0 & \text{ on }\p\Omega.
\end{cases}
\end{equation}
Then we say that:
\begin{enumerate}\setlength{\itemsep}{1mm}
\item A function $\un{u}\in \mathcal{E}^1_m(\Omega)$ is  a {\bf subsolution} to \eqref{psi equation} if:
\[
(\ddbar \un{u})^{m}\wedge\omega^{n-m} \geq \psi^{m}(z,\un{u})\omega^n \ \text{in $\Omega$}.
\]
\item A function $\un{u}\in\mathcal{E}^1_m(\Omega)$ is a {\bf strict subsolution} to \eqref{psi equation} if:
\[
(\ddbar \un{u})^{m}\wedge\omega^{n-m} \geq (\psi(z,\un{u})+\ve_{0})^{m}\omega^{n} \ \text{in $\Omega$}.
\]
for some constant $\ve_{0}>0$.
\item A function $\un{u}\in \mathcal{E}^1_m(\Omega)$ is a {\bf supersolution} to \eqref{psi equation} if:
\[
(\ddbar \ov{u})^{m}\wedge\omega^{n-m} \leq \psi^{m}(z,\ov{u})\omega^n \ \text{in $\Omega$}.
\]
\end{enumerate}
\end{definition}

Since our $\psi$ is general, one cannot conclude that an arbitrary subsolution $\ul{u}$ lies below an arbitrary supersolution $\ov{u}$.

Note that being a supersolution (in the above sense) imposes some additional regularity on $\ov{u}$, as it forces $\Hm(u)$ is to be absolutely continuous with respect to $\omega^n$.

\subsection{Dirichlet problem}

We recall the following important result of Collins-Picard \cite{CP22} on solvability of the Dirichlet problem for compact Hermitian manifolds when the right-hand side does not depend on the solution.

\begin{theorem}[Theorem 1.1 of \cite{CP22}]\label{CP Dirichlet}
Let $(\Omega,\omega)$ be a compact Hermitian manifold with smooth boundary, $\chi\in\Gamma_{m}(\ov\Omega,\omega)$, $f\in C^{\infty}(\ov\Omega)$ a (strictly) positive function, and $\vp\in C^{\infty}(\de\Omega)$. Suppose that there exists $\un{u}\in C^{\infty}(\ov{\Omega})$ such that $\chi+\ddbar\un{u} \in \Gamma_{m}(\ov{\Omega},\omega)$ and
\[
\begin{cases}
(\chi+\ddbar\un{u})^{m}\wedge\omega^{n-m} \geq f^m\omega^{n} & \text{in $\Omega$}, \\
\un{u} = \vp & \text{on $\de\Omega$}.
\end{cases}
\]
Then there exists a unique $u\in C^{\infty}(\ov{\Omega})$ such that $\chi+\ddbar u \in \Gamma_{m}(\ov{\Omega},\omega)$ and
\[
\begin{cases}
(\chi+\ddbar u)^{m}\wedge\omega^{n-m} = f^m\omega^{n} & \text{in $\Omega$}, \\
u = \vp & \text{on $\de\Omega$}.
\end{cases}
\]
\end{theorem}

See \cite{CP22} for definitions. We have the following corollary:

\begin{corollary}\label{CP Dirichlet corollary}
Let $(\Omega,\omega)$ be a strongly $m$-pseudoconvex manifold, $f\in C^{\infty}(\ov\Omega)$ a (strictly) positive function, and $\vp\in C^{\infty}(\de\Omega)$. Then there exists a unique $u\in\mSH({\Omega},\omega)\cap C^{\infty}(\ov{\Omega})$ such that
\begin{equation}\label{CP Dirichlet corollary eqn}
\begin{cases}
(\ddbar u)^{m}\wedge\omega^{n-m} = f^m\omega^{n} & \text{in $\Omega$}, \\
u = \vp & \text{on $\de\Omega$}.
\end{cases}
\end{equation}
\end{corollary}

\begin{proof}
Choose $\chi=\ddbar\rho$, extend $\vp$ to $\ov{\Omega}$ smoothly, and set $\un{u}=\vp+A\rho$. When $A$ is sufficiently large, $\un{u}\in C^{\infty}(\ov{\Omega})$ satisfies  $\chi+\ddbar\un{u} \in \Gamma_{m}(\ov{\Omega},\omega)$ and
\[
\begin{cases}
(\chi+\ddbar\un{u})^{m}\wedge\omega^{n-m} \geq f^m\omega^{n} & \text{in $\Omega$}, \\
\un{u} = \vp & \text{on $\de\Omega$}.
\end{cases}
\]
By Theorem \ref{CP Dirichlet}, there exists a unique $v\in C^{\infty}(\ov{\Omega})$ such that $\chi+\ddbar v \in \Gamma_{m}(\ov{\Omega},\omega)$ and
\[
\begin{cases}
(\chi+\ddbar v)^{m}\wedge\omega^{n-m} = f^m\omega^{n} & \text{in $\Omega$}, \\
v = \vp & \text{on $\de\Omega$}.
\end{cases}
\]
Then $u=\rho+v\in\mSH(\Omega,\omega)$ is the unique solution of \eqref{CP Dirichlet corollary eqn}.
\end{proof}

\subsection{First eigenvalue of the complex Laplacian}

It will be crucial to consider the solution to the linearized eigenvalue problem in our proofs of the main theorems. We briefly recall the results we shall need:

\begin{theorem}\label{first eigenvalue Laplacian}
Let $(\ov\Omega,\alpha)$ be a compact Hermitian manifold with smooth boundary, $\Delta_{\alpha}=\alpha^{i\ov{j}}\de_{i}\de_{\ov{j}}$ the associated complex Laplacian operator, and $f\in C^{\infty}(\ov\Omega)$ a (strictly) positive function. Then:
\begin{enumerate}\setlength{\itemsep}{1mm}
\item There exists a unique pair $(\mu_{1},v_{1})$ such that $\mu_{1}$ is a positive constant, $v_{1}\in C^{\infty}(\ov{\Omega})$ satisfies $\inf_{\Omega}v_{1}=-1$, and
\[
\begin{cases}
\Delta_{\alpha}v_{1} = -\mu_{1}v_{1}f & \text{in $\Omega$}, \\
v_{1} = 0 & \text{on $\de\Omega$}, \\
v_{1} < 0 & \text{in $\Omega$}.
\end{cases}
\]

\item The constant $\mu_{1}$ can be characterized as follows:
\[
\begin{split}
\mu_{1} = \sup\{\mu~|~ & \exists\, v\in C^{\infty}(\ov\Omega),\text{ with } \Delta_{\alpha}v\geq-\mu vf \text{ in }\Omega, \\
& \text{$v=0$ on $\de\Omega$, and $v<0$ in $\Omega$} \}.
\end{split}
\]

\item If $v\in C^{\infty}(\ov\Omega)$ is such that $\Delta_{\alpha}v_{}\geq-\mu_{1}vf$ in $\Omega$, $v=0$ on $\de\Omega$, and $v<0$ in $\Omega$, then there exists a constant $c$ such that $v=c v_{1}$.
\end{enumerate}
\end{theorem}

\begin{proof}
If $\Omega$ is a domain in the Euclidean space, then the result is well-known, see e.g. \cite[Section 6.5.2, Theorem 3]{Ev10}. The same argument also works for general manifolds with smooth boundary; we provide a short sketch.

Let $k=n+2$ and consider the Banach space $X=H^k(\Omega)\cap H_0^1(\Omega)$.
We define the linear, compact operator $A:X\rightarrow X$ given by setting $A(w)=u$, where $u$ is the unique solution to
\begin{equation*}
	\begin{cases}
		\Delta_\alpha u = -wf & \text{in } \Omega  \\
		u=0 & \text{on } \partial\Omega.
	\end{cases}
\end{equation*}
Define also the cone $\mathcal{C}:=\{u\in X~|~u\geq0 \text{ in }\Omega\}$.

Fix a non-zero $w\in\mathcal{C}$, and define $v := A(w)$. By the strong maximum principle and the Hopf Lemma, there exists a constant $\theta>0$ such that $\theta v\geq w$ in $\Omega$.

Let $\varepsilon, \eta>0$ and consider the equation:
\begin{align}
\label{perturbation of linear equ}
	u=\eta A(u+\varepsilon w),
\end{align}
where $u\in\mathcal{C}$ is the unknown. Observe that the existence of a solution to \eqref{perturbation of linear equ} implies that $\eta\leq \theta$; indeed, by the comparison principle, we have:
\[
u\, \geq\, \eta\e A(w)\, \geq\, \frac{\eta}{\theta} \e w.
\]
We can iterate this:
\[
u\, \geq\, \eta \e A(u)\, \geq\, \frac{\eta^2}{\theta}\e A(w)\, \geq\, \frac{\eta^2}{\theta^2} \e w.
\]
Repeating this indefinitely then gives a contradiction unless $\eta \leq \theta$.

Consider now the subspace
\begin{align*}
	S_\varepsilon
	:=\{u\in\mathcal{C}~|~\text{there exists }0\leq\eta\leq2\theta\text{ such that }u=\eta A(u+\varepsilon w)\},
\end{align*}
which is unbounded in $X$ (if not, by  Schaefer's fixed point theorem \cite[Section 9.2.2, Theorem 4]{Ev10}, there exists a solution to \eqref{perturbation of linear equ} with $\eta$ replaced by $2\theta$, a contradiction).

Hence, we may find sequences $0\leq\eta_\varepsilon\leq2\theta$ and $v_\varepsilon\in\mathcal{C}$ with $\|v_\varepsilon\|_X\geq1$ and $v_\varepsilon=\eta_\varepsilon A(v_\e+\varepsilon w)$. Consider the renormalized sequence
\begin{align*}
	u_\varepsilon
	:=\frac{v_\varepsilon}{\|v_\varepsilon\|_X}.
\end{align*}
By the compactness of the operator $A$, there exists a subsequence $\varepsilon_k\rightarrow0$ such that
$\eta_{\varepsilon_k}\rightarrow\eta$ and $u_{\varepsilon_k}\rightarrow u$ in $X$.
Then we have $\|u\|_X=1$.
Since $u_\varepsilon=\eta_\varepsilon A(u_\varepsilon+\frac{\varepsilon w}{\|v_\varepsilon\|_X})$, by taking the limit, we have $u=\eta A(u)$ with $\eta>0$.
By standard elliptic regularity theory, we know that $u$ is smooth. This proves (1) with $\mu_1=\eta$ and $v_1=u$.

Similarly, (3) follows from the proof of \cite[Section 6.5.2, Theorem 3]{Ev10}.

We now show (2). Write $\mu_0:=\sup I$, where
\[
\begin{split}
I= \{\mu~|~ & \exists\, v\in C^{\infty}(\ov\Omega) \text{ with }\Delta_{\alpha}v\geq-\mu vf\text{ in }\Omega, \\
& \text{$v=0$ on $\de\Omega$ and $v<0$ in $\Omega$} \}.
\end{split}
\]
By (1), we have $\mu_1\in I$. Then $\mu_1\leq\mu_0$.
If $v<0$ and $\Delta_\alpha v\geq-\mu vf$ for some $\mu$, then we notice that
\begin{equation*}
	\mu\leq\inf_\Omega\left(-\frac{\Delta_\alpha v}{vf}\right).
\end{equation*}
Thus, we have
\begin{equation*}
	\mu_0=\sup_v\left(\inf_\Omega\left(-\frac{\Delta_\alpha v}{vf}\right)\right),
\end{equation*}
where $v\in C^\infty(\ov\Omega)$ with $v<0$ in $\Omega$ and $v|_{\partial\Omega}=0$.

If $\mu_0>\mu_1$, there exists a $v\in C^{\infty}(\ov\Omega)$ with $v<0$ in $\Omega$ and $v|_{\partial\Omega}=0$ such that
\begin{equation*}
	\inf_\Omega\left(-\frac{\Delta_\alpha v}{vf}\right) >\mu_1,
\end{equation*}
i.e. $\Delta_\alpha v>-\mu_1vf$. Then by (3), there exists a constant $c$ such that $v=cv_1$.
Then we have
\begin{equation*}
	\mu_1<\inf_\Omega\left( -\frac{\Delta_\alpha v}{vf}\right)=\inf_\Omega\left( -\frac{c\Delta v_1}{cv_1f}\right)=\mu_1.
\end{equation*}
This is a contradiction.
\end{proof}

\begin{lemma}\label{max prin c}
Let $\alpha$ be a Hermitian metric on $\ov{\Omega}$ and $(\mu_{1},v_{1})$ the unique eigenvalue of $\Delta_\alpha$, normalized to $\inf_\Omega v_1 = -1$. Suppose that $c\in C^{\infty}(\ov{\Omega})$ is such that $c<\mu_{1}$. If $w\in C^{\infty}(\ov{\Omega})$ satisfies
\[
\begin{cases}
\Delta_{\alpha}w \geq -cw & \text{in $\Omega$}, \\
w = 0 & \text{on $\de\Omega$},
\end{cases}
\]
then $w\leq0$.
\end{lemma}

\begin{proof}
Define the constant $\theta$ by
\[
\theta = \inf\{\gamma~|~w\leq\gamma(-v_{1}) \ \text{in $\Omega$} \}.
\]
Then $w+\theta v_{1}\leq0$ in $\Omega$. It suffices to show $\theta\leq0$. Supposing for the sake of a contradiction that $\theta>0$, we compute
\begin{equation}\label{max prin c eqn 1}
\Delta_{\alpha}(w+\theta v_{1}) \geq -cw-\theta\mu_{1}v_{1} > -c(w+\theta v_{1}) \ \ \text{in $\Omega$}.
\end{equation}
Let $C > 0$ be such that $-C \leq c$. Since $w + \theta v_1 \leq 0$, \eqref{max prin c eqn 1} implies that:
\[
(\Delta_\alpha - C)(w + \theta v_1)\geq 0.
\]
Using the strong maximum principle, either $w+\theta v_{1}\equiv0$ or $w+\theta v_{1}<0$ in $\Omega$. By \eqref{max prin c eqn 1}, the former case cannot happen. In the latter case, by the Hopf lemma, there exists $\ve>0$ such that
\[
w+\theta v_{1} \leq \ve v_{1} \ \ \text{in $\Omega$},
\]
contradicting the definition of $\theta$.
\end{proof}

\subsection{Stability estimate}\label{stability subsection}

As mentioned in the introduction, we will make critical usage of the stability theorem of Dinew-Ko\l odziej \cite{DK14}; it can be shown for strongly $m$-pseudoconvex manifolds in the same way as for domains, so we only sketch the proof. We need the following standard $L^\infty$ bound for smooth solutions:

\begin{theorem}\label{zero order estimate}
Let $(\ov\Omega,\omega)$ be a compact strongly $m$-pseudoconvex K\"ahler manifold with smooth boundary, and let $p > \frac{n}{m}$. Suppose that $u\in\mSH(\Omega,\omega)\cap C^{\infty}(\ov{\Omega})$ satisfies
\[
\begin{cases}
(\ddbar u)^{m}\wedge\omega^{n-m} = f\omega^{n} & \text{in $\Omega$}, \\
u = \vp & \text{on $\de\Omega$}.
\end{cases}
\]
for some (strictly) positive function $f\in C^{\infty}(\ov\Omega)$ and $\vp\in C^{\infty}(\de\Omega)$. Then there exists a constant depending only on $p$, $m$, $n$ and $(\ov\Omega,\omega)$ such that
\[
\|u\|_{L^{\infty}(\Omega)} \leq \|\vp\|_{L^{\infty}(\de\Omega)}+C\|f\|_{L^{p}(\Omega)}^{1/m}.
\]
\end{theorem}

Theorem \ref{zero order estimate} can be proved exactly as for domains, by using Dinew-Ko\l odziej's volume-capacity estimate \cite{DK14}.

\begin{theorem}[Theorem 2.8 of \cite{DK14}]\label{stability estimate}
For $i=1,2$, suppose that $u_{i}\in\mSH(\Omega,\omega)\cap C^{\infty}(\ov\Omega)$ satisfies
\[
\begin{cases}
(\ddbar u_{i})^{m}\wedge\omega^{n-m} = f_{i}^m\omega^{n} & \text{in $\Omega$}, \\
u_{i} = \vp_{i} & \text{on $\de\Omega$}.
\end{cases}
\]
for some (strictly) positive function $f_{i}\in C^{\infty}(\ov\Omega)$ and $\vp_{i}\in C^{\infty}(\de\Omega)$. Then for any $p>n/m$, there exists a constant $C$ depending only on $p$ and $(\Omega,\omega)$ such that
\[
\|u_{1}-u_{2}\|_{L^{\infty}(\Omega)} \leq \|\vp_{1}-\vp_{2}\|_{L^{\infty}(\de\Omega)}+C\|f_{1}^m-f_{2}^m\|_{L^{p}(\Omega)}^{1/m}.
\]
\end{theorem}

\begin{proof}
By Corollary \ref{CP Dirichlet corollary}, there exists  $w_{\ve}\in\mSH(\ov{\Omega},\omega)\cap C^{\infty}(\ov{\Omega})$ such that:
\[
\begin{cases}
\Hm(w_{\ve}) = \sqrt{(f_1^m-f_2^m)^2+\e^2}\omega^{n} & \text{in $\Omega$}, \\
w_{\ve} = -\|\vp_{1}-\vp_{2}\|_{L^{\infty}(\de\Omega)} & \text{on $\de\Omega$}.
\end{cases}
\]
Then the domination principle (Proposition \ref{Domination Principle}), shows that $u_1 + w_\e \leq u_2$, since:
\[
\begin{cases}
\Hm(u_{1}+w_{\ve}) \geq \Hm(u_{1})+\Hm(w_{\ve}) \geq \Hm(u_{2}) & \text{in $\Omega$}, \\
u_{1}+w_{\ve} \leq u_{2} & \text{on $\de\Omega$}.
\end{cases}
\]
We conclude that:
\[
u_1 - u_2 \leq \norm{w_\e}_{L^\infty(\Omega)} \leq \|\vp_{1}-\vp_{2}\|_{L^{\infty}(\de\Omega)}+C\norm{\sqrt{(f_1^m-f_2^m)^2 + \e^2}}_{L^p(\Omega)}^{1/m},
\]
by Theorem \ref{zero order estimate}. Letting $\e\rightarrow 0$ and switching the roles of $u_1, u_2$ finishes.
\end{proof}

An immediate corollary of this and the continuity of the complex Hessian operator is the existence and uniqueness of a continuous solution $u\in\mSH(\Omega)\cap C^0(\ov{\Omega})$ to the equation:
\[
\begin{cases}
\Hm(u) = f^m\omega^n & \text{ in }\Omega,\\
u = \vp &\text{ on }\p\Omega,
\end{cases}
\]
when $0 \leq f\in L^p(\Omega)$ with $p > n$ and $\vp\in C^{0}(\de\Omega)$.

\subsection{Finite Energy Classes}\label{subsection: energy}

We recall some facts about finite energy classes of $m$-subharmonic functions, see e.g. \cite{BBGZ13, Ceg98, Lu15}. These results are standard for strongly $m$-pseudoconvex domains in $\C^n$, and the adaptations to the manifold case are straightforward; we have collected some further details in the Appendix.

If $u\in \mSH(\Omega)\cap L^\infty(\Omega)$, then the Chern-Levine-Nirenburg inequalities imply that:
\[
\Hm(u) := (\ddbar u)^m\wedge \omega^{n-m}
\]
is a well-defined Radon measure on $\Omega$.

\begin{definition}
We define the class $\mathcal{E}^0_m(\Omega)$ to be the class of $m$-subharmonic functions $u\in \mSH(\Omega)\cap L^\infty(\Omega)$ which additionally satisfy $\int_\Omega \Hm(u) < \infty$ and
\[
\lim_{x\rightarrow z} u(x) = 0\text{ for all }z\in \p\Omega.
\]
\end{definition}
One can show that if $u_j\in\mathcal{E}^0_m(\Omega)$ decrease to $u\in\mathcal{E}^0_m(\Omega)$ as $j\rightarrow\infty$, then the measures $\Hm(u_j)$ weakly converge to $\Hm(u)$, following \cite{BT76, Ceg98, Lu15}.

We then define the finite energy class $\mathcal{E}^1_m(\Omega)$ as follows:
\begin{definition}
Suppose $u\in\mSH(\Omega)$. Then we say that $u\in \mathcal{E}^1_m(\Omega)$, the finite energy class, if there exists a decreasing sequence $u_j\in\mathcal{E}^0_m(\Omega)$ such that:
\[
\lim_{j\rightarrow\infty} u_j = u\ \text{ and }\ \lim_{j\rightarrow \infty} \int_\Omega (-u_j)\Hm(u_j) =: (m+1)E_m(u) < \infty.
\]
\end{definition}

If $u\in\mathcal{E}^1_m(\Omega)$ and $u_j\in \mathcal{E}^0_m(\Omega)$ is any sequence decreasing to $u$, then it is standard \cite{Lu15} to show that:
\[
\lim_{j\rightarrow\infty} E_m(u_j) = E_m(u).
\]
Moreover, the equality $E_m(u) = \frac{1}{m+1}\int_\Omega (-u)\Hm(u)$ holds.

We conclude with some energy estimates which play an important role in Section \ref{Existence and Uniqueness of the Eigenfunction}. The first is due to B\l ocki \cite{Bl93}:
\begin{proposition}\label{Poincare type inequality proposition}
Suppose that $w\in\mathcal{E}^1_m(\Omega), v\in \mathcal{E}^0_m(\Omega)$. Then:
 \begin{equation}
	\label{Blocki ineq}
		\int_\Omega(-w)^{m+1} (\ddbar v)^m\wedge\omega^{n-m}
		\leq(m+1)!\|v\|_{L^\infty}^m E_m(w);
\end{equation}
\end{proposition}
\begin{proof}
The proof follows from an integration by parts; indeed, for any $1 \leq p \leq m$, we have:
\[
\begin{split}
& \int_\Omega(-w)^{p+1}(\ddbar w)^{m-p}\wedge(\ddbar v)^p\wedge\omega^{n-m}\\
= {} & \int_\Omega v\ddbar(-w)^{p+1}\wedge(\ddbar w)^{m-p}\wedge(\ddbar v)^{p-1}\wedge\omega^{n-m}.
\end{split}
\]
Since $\ddbar(-w)^{p+1}\geq -(p+1)(-w)^{p}\ddbar w,$ we can apply induction to conclude the argument. The validity of these operations can be shown in the same way as the domain case, see e.g. \cite{Lu15}.
\end{proof}

\subsection{Subextenstion Theorem}

For the proof of Theorem \ref{Theorem 3}, we need the following result of Cegrell-Ko\l odziej-Zeriahi \cite{CKZ11} and Pham \cite{Pham08}:
\begin{theorem}\label{Subextension Theorem}
Suppose that $\Omega'\subset \Omega$ are strongly $m$-pseudoconvex manifolds, with $\omega$ a K\"ahler form on $\Omega$. Then for any $w'\in \mathcal{E}^1_m(\Omega')$, there exists some $w\in\mathcal{E}^1_m(\Omega)$ such that:
\[
w \leq w'\text{ in }\Omega,
\]
and
\[
E_{m, \Omega}(w) \leq E_{m, \Omega'}(w').
\]
Here we write $E_{m,\Omega}(w)$ for the energy of $w$ over $\Omega$.
\end{theorem}

\begin{proof}
The results in \cite{CKZ11, Pham08} are only written for the case $m = n$, but the proof carry over directly to our setting.
\end{proof}

\section{A Priori Estimates}
\label{A Priori Estimates}
In this section, we prove {\it a priori} estimates for the complex $m$-Hessian equation, which will play an important role in the next section.

Let $\psi$ be a smooth (strictly) positive function on $\ov{\Omega}\times(-\infty,0]$, and suppose that $u\in\mSH(\Omega)\cap C^{\infty}(\ov\Omega)$ satisfies
\begin{equation}\label{Dirichlet psi 1}
\begin{cases}
(\ddbar u)^{m}\wedge\omega^{n-m} = \psi^{m}(z,v)\omega^{n} & \text{in $\Omega$}, \\
u = 0 & \text{on $\de\Omega$},
\end{cases}
\end{equation}
for some $v\in C^{\infty}(\ov\Omega)$. Suppose that there exist $\un{u}\in\mSH(\Omega)\cap C^{\infty}(\ov{\Omega})$ and $w\in\SH_{1}(\Omega)\cap C^{\infty}(\ov{\Omega})$ such that
\begin{equation}\label{un u u w condition}
\un{u} \leq u \leq w < 0 \ \text{in \text{$\Omega$}}, \quad
\un{u} = u = w = 0 \ \text{on \text{$\de\Omega$}}.
\end{equation}

The goal of this section is to prove the following theorem:
\begin{theorem}\label{a priori estimates}
Under the above assumptions, there exists a constant $C$ depending only on $\|\psi\|_{C^{2}}$, $\|\un{u}\|_{C^{2}}$, $\|w\|_{C^{1}}$ and $(\ov\Omega,\omega)$ such that:
\begin{equation}\label{C 0 C 1 boundary}
\sup_{\Omega}|u|+\sup_{\de\Omega}|\de u| \leq C
\end{equation}
and
\begin{equation}\label{C 2 estimates}
\sup_{\ov{\Omega}}|\nabla^{2}u| \leq \frac{1}{2}\sup_{\Omega}|\nabla^{2}v|+C\sup_{\Omega}|\de u|^{2}+C\sup_{\Omega}|\de v|^{2}+C.
\end{equation}
\end{theorem}

For convenience, by replacing $\psi^{m}$ with $\binom{n}{m}^{-1}\psi^{m}$, we can rewrite \eqref{Dirichlet psi 1} as:
\begin{equation}\label{Dirichlet psi 2}
\begin{cases}
\sigma_{m}(u) = \psi^{m}(z,v) & \text{in $\Omega$}, \\
u = 0 & \text{on $\de\Omega$}.
\end{cases}
\end{equation}

It it clear that \eqref{C 0 C 1 boundary} follows (more or less) immediately from \eqref{un u u w condition}. We thus must show \eqref{C 2 estimates}, which we do in two parts. We first show the interior Hessian estimates in Proposition \ref{interior Hessian estimate}; we then show the boundary Hessian estimates in Proposition \ref{boundary Hessian estimate}. For notational convenience, we shall say that a constant is uniform if it depends only on $\|\psi\|_{C^{2}}$, $\|\un{u}\|_{C^{2}}$, $\|w\|_{C^{1}}$ and $(\ov\Omega,\omega)$; we will always write $C$ for such a constant, whose exact value may change from line to line.

\subsection{Interior Hessian estimate}
\begin{proposition}\label{interior Hessian estimate}
Under the same assumptions of Theorem \ref{a priori estimates}, there exists a uniform constant $C$ such that
\begin{equation}\label{interior second estimate}
\sup_{\Omega}|\nabla^{2}u| \leq C\sup_{\de\Omega}|\nabla^{2}u|+\frac{1}{2}\sup_{\Omega}|\nabla^{2}v|+C\sup_{\Omega}|\de u|^{2}+C\sup_{\Omega}|\de v|^{2}+C.
\end{equation}
\end{proposition}

\begin{proof}
We shall use the techniques of \cite{CTW19,CM21,CHZ23}. Let $g$ be the Riemannian metric corresponding to $\omega$, and write $\lambda_{1}(\nabla^{2}u)\geq\ldots\geq\lambda_{2n}(\nabla^{2}u)$ for the eigenvalues of $\nabla^{2}u$ with respect to $g$ (not to be confused with $\lambda_1(\Omega, f)$, which is not used in this proof). Denote the Laplace-Beltrami operator of $g$ by $\Delta_{\mathbb{R}}$ and write
\[
\Delta_{\omega}u = \frac{n\ddbar u\wedge\omega^{n-1}}{\omega^{n}}.
\]
Using Maclaurin's inequality and \eqref{Dirichlet psi 2},
\[
\sum_{\alpha=1}^{2n}\lambda_{\alpha}(\nabla^{2}u) = \Delta_{\mathbb{R}}u = 2\Delta_{\omega}u
\geq 2n\left[\binom{n}{m}^{-1}\sigma_{m}(u)\right]^{\frac{1}{m}}
= 2n\binom{n}{m}^{-\frac{1}{m}}\psi > 0.
\]
This shows $\lambda_{1}(\nabla^{2}u)>0$ and
\begin{equation}\label{D 2 u lambda 1}
|\nabla^{2}u| \leq C_{n}\lambda_{1}(\nabla^{2}u)
\end{equation}
for some constant $C_{n}$ depending only on $n$. Set
\[
K = \sup_{\Omega}|\de u|^{2}+\sup_{\Omega}|\de v|^{2}+1, \ \
N = \sup_{\Omega}|\nabla^{2}u|+1, \ \
\theta = \nabla^{2}u+Ng.
\]
We consider the quantity
\[
Q = \log\lambda_{1}(\nabla^{2}u)+\xi(|\theta|^{2})+\eta(|\de u|^{2})+A\rho,
\]
where
\[
\xi(s) = -\frac{1}{6}\log(5N^{2}-s), \ \
\eta(s) = -\frac{1}{6}\log(2K-s),
\]
and $A$ is a uniform constant to be determined later.

The first eigenvalue $\lambda_{1}(\nabla^{2}u)$ is a continuous function in $\ov\Omega$. Assume that $x_{0}$ is a maximum point of $Q$. By the definition of $Q$ and the fact that $\rho\leq0$, we see that
\begin{equation}\label{Q x 0}
Q(x_{0}) \leq \log\lambda_{1}(\nabla^{2}u)(x_{0})-\frac{1}{6}\log N^{2}-\frac{1}{6}\log K.
\end{equation}
We assume without loss of generality that $\sup_{\Omega}|\nabla^{2}u|\geq1$ (i.e. $N\geq2$), and let $y_{0}$ be the maximum point of $|\nabla^{2}u|$. Then \eqref{D 2 u lambda 1} shows
\[
\lambda_{1}(\nabla^{2}u)(y_{0} )\geq \frac{|\nabla^{2}u(y_{0})|}{C_{n}}
= \frac{N-1}{C_{n}} \geq \frac{N}{2C_{n}}
\]
and so
\begin{equation}\label{Q y 0}
Q(y_{0}) \geq \log\frac{N}{2C_{n}}-\frac{1}{6}\log(5N^{2})-\frac{1}{6}\log(2K)-A\|\rho\|_{C^{0}}.
\end{equation}
Combining \eqref{Q x 0} and \eqref{Q y 0} with $Q(y_{0})\leq Q(x_{0})$, we see that
\begin{equation}\label{N lambda 1}
N \leq e^{C_{0}A}\lambda_{1}(\nabla^{2}u)(x_{0})
\end{equation}
for some uniform constant $C_{0}$. Hence, to prove \eqref{interior second estimate}, it suffices to prove
\begin{equation}\label{goal}
\lambda_{1}(\nabla^{2}u)(x_{0}) \leq \sup_{\de\Omega}|\nabla^{2}u|+\frac{1}{2e^{C_{0}A}}\sup_{\Omega}|\nabla^{2}v|+C_{A}K
\end{equation}
for some uniform constant $C_{A}$ depending on $A$. In the following argument, we always assume that $x_{0}\notin\de\Omega$ (otherwise there is nothing to show) and that
\begin{equation}\label{WLOG assumption}
\lambda_{1}(\nabla^{2}u)(x_{0}) > \frac{1}{2e^{C_{0}A}}\sup_{\Omega}|\nabla^{2}v|+K.
\end{equation}

Choose a holomorphic normal coordinate system $(U,\{z^{i}\}_{i=1}^{n})$ for $g$ centered at $x_{0}$. Writing $z^{i}=x^{2i-1}+\sqrt{-1}x^{2i}$, we have that $(U,\{x^{\alpha}\}_{\alpha=1}^{2n})$ is a real coordinate system near $x_{0}$. After rotating the coordinates, we assume that
\[
u_{i\ov{j}} = \delta_{ij}u_{i\ov{i}}, \ \ u_{1\ov{1}} \geq u_{2\ov{2}} \geq \cdots \geq u_{n\ov{n}}, \ \ \text{at $x_{0}$}.
\]
Let $\Lambda_{1}\geq\Lambda_{2}\geq\ldots\geq\Lambda_{n}$ be the eigenvalues of $\ddbar u$ with respect to $\omega$. Then $\sigma_{m}^{1/m}(u)=\sigma_{m}^{1/m}(\Lambda)$. Let $F^{i\ov{j}}$ and $F^{i\ov{j},k\ov{l}}$ be the first and second derivatives of $\sigma_{m}^{1/m}$, i.e.
\[
F^{i\ov{j}} = \frac{\de \sigma_{m}^{1/m}}{\de u_{i\ov{j}}}, \ \
F^{i\ov{j},k\ov{l}} = \frac{\de^{2}\sigma_{m}^{1/m}}{\de u_{i\ov{j}}\de u_{k\ov{l}}}.
\]
Then at $x_{0}$, we have (see e.g. \cite{EH89,Spr05})
\[
F^{i\ov{j}} = \delta_{ij}F^{i\ov{i}} = \frac{\de\sigma_{m}^{1/m}}{\de\Lambda_{i}}, \ \ F^{1\ov{1}} \leq F^{2\ov{2}} \leq \cdots \leq F^{n\ov{n}}
\]
and (see e.g. \cite{A94,G96,Spr05})
\[
F^{i\ov{j},k\ov{l}} = \frac{\de^{2}\sigma_{m}^{1/m}}{\de\Lambda_{i}\de\Lambda_{k}}\delta_{ij}\delta_{kl}
+\frac{F^{i\ov{i}}-F^{j\ov{j}}}{\Lambda_{i}-\Lambda_{j}}
(1-\delta_{ij})\delta_{il}\delta_{jk},
\]
where the quotient is interpreted as a limit if $\Lambda_{i}=\Lambda_{j}$.

When $\lambda_{1}(x_{0})=\lambda_{2}(x_{0})$, the function $\lambda_{1}(\nabla^{2}u)$ might not be smooth. To avoid such this, we apply a standard perturbation argument (see e.g. \cite{STW17,Sze18,CTW19}). For $1\leq\alpha\leq 2n$, let $V_{\alpha}=V_{\alpha}^{\beta}\de_{\beta}$ be the $g$-unit eigenvector of $\lambda_{\alpha}$ at $x_{0}$, and extend $V_{\alpha}$ to a vector field in $U$ by taking the components $V_{\alpha}^{\beta}$ to be constant. Define
\[
\Phi = \Phi_{\beta}^{\alpha} \frac{\de}{\de x^{\alpha}}\otimes dx^{\beta}
= g^{\alpha\gamma}\left[u_{\gamma\beta}-(\delta_{\gamma\beta}-V_{1}^{\gamma}V_{1}^{\beta})\right] \frac{\de}{\de x^{\alpha}}\otimes dx^{\beta}.
\]
Let $\lambda_{1}(\Phi)\geq\lambda_{2}(\Phi)\geq\ldots\geq\lambda_{2n}(\Phi)$ be the eigenvalues of $\Phi$ with respect to $g$. It is clear that $\lambda_{1}(\Phi)(x_{0})>\lambda_{2}(\Phi)(x_{0})$, $\lambda_{1}(\nabla^{2}u)(x_{0})=\lambda_{1}(\Phi)(x_{0})$ and $\lambda_{1}(\nabla^{2}u)\geq\lambda_{1}(\Phi)$ in $U$. Define the perturbed quantity in $U$:
\[
\hat{Q} = \log\lambda_{1}(\Phi)+\xi(|\theta|^{2})+\eta(|\de u|^{2})+A\rho.
\]
Then $\hat{Q}$ is smooth at $x_{0}$ and $x_{0}$ is still a maximum point of $\hat{Q}$. For convenience, we denote $\lambda_{\alpha}(\Phi)$ by $\lambda_{\alpha}$ in the following argument.

\begin{lemma}
At $x_{0}$, we have
\begin{equation}\label{maximum principle equality}
\frac{u_{V_{1}V_{1}i}}{\lambda_{1}} = -\xi'(|\theta|^{2})_{i}-\eta'(|\partial u|^{2})_{i}-A\rho_{i}
\end{equation}
and
\begin{equation}\label{maximum principle inequality}
\begin{split}
0 \geq {} & G_{1}+G_{2}+G_{3}-B +\xi'' F^{i\ov{i}}|(|\theta|^{2})_{i}|^{2}+\eta'' F^{i\ov{i}}|(|\partial u|^{2})_{i}|^{2} \\
& + \frac{1}{8K}\sum_{i,j}F^{i\ov{i}}(|u_{ij}|^{2}+|u_{i\ov{j}}|^{2})+(A-C)\mathcal{F}-C_{A},
\end{split}
\end{equation}
where
\[
G_1 := 2\sum_{\alpha>1}\frac{F^{i\ov{i}}|u_{V_{1}V_{\alpha}i}|^2}{\lambda_{1}(\lambda_{1}-\lambda_{\alpha})}, \quad
G_2:= -\frac{1}{\lambda_{1}}\sum_{i\neq k}F^{i\ov{k},k\ov{i}}|u_{i\ov{k}V_{1}}|^{2},
\]
\[
G_3 := \sum_{\alpha,\beta}\frac{F^{i\ov{i}}|u_{\alpha\beta i}|^{2}}{C_{A}\lambda_{1}^{2}} ,\quad
B := \frac{F^{i\ov{i}}|u_{V_{1}V_{1}i}|^{2}}{\lambda_{1}^{2}}, \quad
\mathcal{F} := \sum_{i}F^{i\ov{i}}.
\]
\end{lemma}

\begin{proof}
We first recall the formulas for the first and second derivatives of $\lambda_{1}$ (see e.g. \cite[Lemma 5.2]{CTW19}):
\begin{equation*}
\begin{split}
\frac{\partial \lambda_{1}}{\partial \Phi^{\alpha}_{\beta} }
= {} & V_{1}^{\alpha}V_{1}^{\beta}, \\
\frac{\partial^{2} \lambda_{1}}{\partial \Phi^{\alpha}_{\beta}\partial \Phi^{\gamma}_{\delta}}
= {} & \sum_{\mu>1}\frac{V_{1}^{\alpha}V_{\mu}^{\beta}V_{\mu}^{\gamma}V_{1}^{\delta} +V_{\mu}^{\alpha}V_{1}^{\beta}V_{1}^{\gamma}V_{\mu}^{\delta}}{\lambda_{1}-\lambda_{\mu}}.
\end{split}
\end{equation*}
Using $\hat{Q}_{i}=0$ and $F^{i\ov{i}}Q_{i\ov{i}}\leq0$ at $x_{0}$, we obtain \eqref{maximum principle equality} and
\begin{equation}\label{maximum principle inequality 1}
\begin{split}
0 \geq {} & \frac{F^{i\ov{i}}(\lambda_{1})_{i\ov{i}}}{\lambda_{1}}-\frac{F^{i\ov{i}}|u_{V_{1}V_{1}i}|^{2}}{\lambda_{1}^{2}}
+\xi' F^{i\ov{i}}(|\theta|^{2})_{i\ov{i}}+\xi''F^{i\ov{i}}|(|\theta|^{2})_{i}|^{2} \\
& +\eta'F^{i\ov{i}}(|\partial u|^{2})_{i\ov{i}}+\eta'' F^{i\ov{i}}|(|\partial u|^{2})_{i}|^{2}+AF^{i\ov{i}}\rho_{i\ov{i}}.
\end{split}
\end{equation}

For the first term of \eqref{maximum principle inequality 1}, we compute
\[
\begin{split}
F^{i\ov{i}}(\lambda_{1})_{i\ov{i}}
= {} & F^{i\ov{i}}\lambda_{1}^{\alpha\beta,\gamma\delta}(\Phi_{\beta}^{\alpha})_{i}(\Phi_{\delta}^{\gamma})_{\ov{i}}
+F^{i\ov{i}}\lambda_{1}^{\alpha\beta}(\Phi_{\alpha}^{\beta})_{i\ov{i}} \\[3.5mm]
= {} & F^{i\ov{i}}\lambda_{1}^{\alpha\beta,\gamma\delta}u_{\alpha\beta i}u_{\gamma\delta\ov{i}}
+F^{i\ov{i}}\lambda_{1}^{\alpha\beta}u_{\alpha\beta i\ov{i}} \\[2mm]
= {} & 2\sum_{\alpha>1}\frac{F^{i\ov{i}}|u_{V_{1}V_{\alpha}i}|^2}{\lambda_{1}-\lambda_{\alpha}}
+F^{i\ov{i}}u_{V_{1}V_{1}i\ov{i}} \\
\geq {} & 2\sum_{\alpha>1}\frac{F^{i\ov{i}}|u_{V_{1}V_{\alpha}i}|^2}{\lambda_{1}-\lambda_{\alpha}}
+F^{i\ov{i}}u_{i\ov{i}V_{1}V_{1}}-C\lambda_{1}\mathcal{F}.
\end{split}
\]
Using \eqref{Dirichlet psi 2}, we have
\begin{equation}\label{sigma m equation}
\sigma_{m}^{1/m}(u) = \psi.
\end{equation}
Applying $\nabla_{V_{1}}\nabla_{V_{1}}$ to \eqref{sigma m equation} and using \eqref{WLOG assumption} as well as $\lambda_{1}(\nabla^{2}u)(x_{0})=\lambda_{1}(\Phi)(x_{0})$,
\[
\begin{split}
F^{i\ov{i}}u_{i\ov{i}V_{1}V_{1}}
= {} & -F^{i\ov{j},k\ov{l}}u_{i\ov{j}V_{1}}u_{k\ov{l}V_{1}}
+\psi_{v}v_{V_{1}V_{1}}+\psi_{vv}v_{V_{1}}^{2}+2\psi_{V_{1}v}v_{V_{1}}+\psi_{V_{1}V_{1}} \\[2mm]
\geq {} & -F^{i\ov{j},k\ov{l}}u_{i\ov{j}V_{1}}u_{k\ov{l}V_{1}}
-C\sup_{\Omega}|\nabla^{2}v|-CK \\
\geq {} & -F^{i\ov{j},k\ov{l}}u_{i\ov{j}V_{1}}u_{k\ov{l}V_{1}}
-2Ce^{C_{0}A}\lambda_{1}-C\lambda_{1}.
\end{split}
\]
Then
\[
\frac{F^{i\ov{i}}(\lambda_{1})_{i\ov{i}}}{\lambda_{1}}
\geq 2\sum_{\alpha>1}\frac{F^{i\ov{i}}|u_{V_{1}V_{\alpha}i}|^2}{\lambda_{1}(\lambda_{1}-\lambda_{\alpha})}
-\frac{1}{\lambda_{1}} F^{i\ov{j},k\ov{l}}u_{i\ov{j}V_{1}}u_{k\ov{l}V_{1}}-C\mathcal{F}-3Ce^{C_{0}A}.
\]
Thanks to the concavity of $\sigma_{m}^{1/m}$, we have
\[
\begin{split}
-\frac{1}{\lambda_{1}} F^{i\ov{j},k\ov{l}}u_{i\ov{j}V_{1}}u_{k\ov{l}V_{1}}
= {} & -\frac{1}{\lambda_{1}}F^{i\ov{i},k\ov{k}}u_{i\ov{i}V_{1}}u_{k\ov{k}V_{1}}-\frac{1}{\lambda_{1}}\sum_{i\neq k}F^{i\ov{k},k\ov{i}}|u_{i\ov{k}V_{1}}|^{2} \\
\geq {} & -\frac{1}{\lambda_{1}}\sum_{i\neq k}F^{i\ov{k},k\ov{i}}|u_{i\ov{k}V_{1}}|^{2}
\end{split}
\]
and so
\begin{equation}\label{first term}
\frac{F^{i\ov{i}}(\lambda_{1})_{i\ov{i}}}{\lambda_{1}}
\geq 2\sum_{\alpha>1}\frac{F^{i\ov{i}}|u_{V_{1}V_{\alpha}i}|^2}{\lambda_{1}(\lambda_{1}-\lambda_{\alpha})}
-\frac{1}{\lambda_{1}}\sum_{i\neq k}F^{i\ov{k},k\ov{i}}|u_{i\ov{k}V_{1}}|^{2}-C\mathcal{F}-Ce^{C_{0}A}.
\end{equation}

For the third term of \eqref{maximum principle inequality 1},
\[
\begin{split}
F^{i\ov{i}}(|\theta|^{2})_{i\ov{i}}
= {} & 2\sum_{\alpha,\beta}F^{i\ov{i}}|u_{\alpha\beta i}|^{2}+2\sum_{\alpha,\beta}\theta_{\alpha\beta}F^{i\ov{i}}u_{\alpha\beta i\ov{i}} \\
\geq {} & 2\sum_{\alpha,\beta}F^{i\ov{i}}|u_{\alpha\beta i}|^{2}+2\sum_{\alpha,\beta}\theta_{\alpha\beta}F^{i\ov{i}}u_{i\ov{i}\alpha\beta}
-CN^{2}\mathcal{F}.
\end{split}
\]
Applying $\nabla_{\alpha}\nabla_{\beta}$ to \eqref{sigma m equation} and using \eqref{WLOG assumption} as well as $\lambda_{1}(\nabla^{2}u)(x_{0})=\lambda_{1}(x_{0})$,
\[
\begin{split}
& 2\sum_{\alpha,\beta}\theta_{\alpha\beta}F^{i\ov{i}}u_{i\ov{i}\alpha\beta} \\
\geq {} & 2\sum_{\alpha,\beta}\theta_{\alpha\beta}
\left(-F^{i\ov{j},k\ov{l}}u_{i\ov{j}\alpha}u_{k\ov{l}\beta}+\psi_{v}v_{\alpha\beta}+\psi_{vv}v_{\alpha}v_{\beta}
+\psi_{\alpha v}v_{\beta}+\psi_{\beta v}v_{\alpha}+\psi_{\alpha\beta}\right) \\
\geq {} & -2\sum_{\alpha,\beta}\theta_{\alpha\beta}F^{i\ov{j},k\ov{l}}u_{i\ov{j}\alpha}u_{k\ov{l}\beta}-CN\sup_{\Omega}|\nabla^{2}v|-CKN \\
\geq {} & -2\sum_{\alpha,\beta}\theta_{\alpha\beta}F^{i\ov{j},k\ov{l}}u_{i\ov{j}\alpha}u_{k\ov{l}\beta}-2Ce^{C_{0}A}N\lambda_{1}-CN\lambda_{1}.
\end{split}
\]
Using $\theta>0$ and the concavity of $\sigma_{m}^{1/m}$, we see that
\[
-2\sum_{\alpha,\beta}\theta_{\alpha\beta}F^{i\ov{j},k\ov{l}}u_{i\ov{j}\alpha}u_{k\ov{l}\beta} \geq 0,
\]
and then
\[
F^{i\ov{i}}(|\theta|^{2})_{i\ov{i}}
\geq 2\sum_{\alpha,\beta}F^{i\ov{i}}|u_{\alpha\beta i}|^{2}-CN^{2}\mathcal{F}-3Ce^{C_{0}A}N\lambda_{1}.
\]
Using $\frac{1}{30N^{2}}\leq\xi'\leq\frac{1}{6N^{2}}$ and \eqref{N lambda 1} as well as $\lambda_{1}(\nabla^{2}u)(x_{0})=\lambda_{1}(\Phi)(x_{0})$,
\begin{equation}\label{third term}
\begin{split}
\xi' F^{i\ov{i}}(|\theta|^{2})_{i\ov{i}}
\geq {} & \frac{1}{15N^{2}}\sum_{\alpha,\beta}F^{i\ov{i}}|u_{\alpha\beta i}|^{2}-C\mathcal{F}-Ce^{C_{0}A} \\
\geq {} & \frac{1}{15e^{2C_{0}A}\lambda_{1}}\sum_{\alpha,\beta}F^{i\ov{i}}|u_{\alpha\beta i}|^{2}-C\mathcal{F}-Ce^{C_{0}A}.
\end{split}
\end{equation}
For the fifth term of \eqref{maximum principle inequality 1},
\[
\begin{split}
F^{i\ov{i}}(|\partial u|^{2})_{i\ov{i}}
= {} & \sum_{j}F^{i\ov{i}}(|u_{ij}|^{2}+|u_{i\ov{j}}|^{2})+2\mathrm{Re}\left(F^{i\ov{i}}u_{ki\ov{i}}u_{\ov{k}}\right) \\
\geq {} & \sum_{j}F^{i\ov{i}}(|u_{ij}|^{2}+|u_{i\ov{j}}|^{2})+2\mathrm{Re}\left(F^{i\ov{i}}u_{i\ov{i}k}u_{\ov{k}}\right)-CK\mathcal{F}.
\end{split}
\]
Applying $\nabla_{k}$ to \eqref{sigma m equation},
\[
2\mathrm{Re}\left(F^{i\ov{i}}u_{i\ov{i}k}u_{\ov{k}}\right)
= 2\mathrm{Re}\left((\psi_{v}v_{k}+\psi_{k})u_{\ov{k}}\right) \geq -CK.
\]
Using $\frac{1}{12K}\leq\eta'\leq\frac{1}{6K}$,
\begin{equation}\label{fifth term}
\eta'F^{i\ov{i}}(|\partial u|^{2})_{i\ov{i}}
\geq \frac{1}{12K}\sum_{j}F^{i\ov{i}}(|u_{ij}|^{2}+|u_{i\ov{j}}|^{2})-C\mathcal{F}-C.
\end{equation}

For the last term of \eqref{maximum principle inequality 1}, since $\ddbar\rho$ is strictly $m$-positive, then there exists $\delta_{0}>0$ such that $\ddbar\rho-\delta_{0}\omega$ is still $m$-positive. By G{\aa}rding's inequality,
\[
\sum_{i}F^{i\ov{i}}(\rho_{i\ov{i}}-\delta_{0}) > 0,
\]
which shows
\[
F^{i\ov{i}}\rho_{i\ov{i}} = \sum_{i}F^{i\ov{i}}(\rho_{i\ov{i}}-\delta_{0})+\delta_{0}\sum_{i}F^{i\ov{i}} > \delta_{0}\mathcal{F}.
\]
Replacing $\rho$ by $\delta_{0}^{-1}\rho$ if necessary, we assume without loss of generality that $\delta_{0}=1$. Then
\begin{equation}\label{last term}
AF^{i\ov{i}}\rho_{i\ov{i}} \geq A\mathcal{F}.
\end{equation}
Substituting \eqref{first term}, \eqref{third term}, \eqref{fifth term} and \eqref{last term} into \eqref{maximum principle inequality 1}, we obtain \eqref{maximum principle inequality}.
\end{proof}

\subsubsection{The term $B$}

We now need to deal with the term $B$, which we shall do by further splitting it into several smaller terms, which can individually be controlled by the other good terms. Define
\[
S = \{1\leq i\leq n-1~|~F^{i\ov{i}}\leq A^{-4}F^{i+1\ov{i+1}}\},
\]
and
\[
i_{0} =
\begin{cases}
\ 0 & \mbox{if $S=\emptyset$}, \\
\ \max_{i\in S}i & \mbox{if $S\neq\emptyset$},
\end{cases}
\quad \quad I = \{i_{0}+1,\ldots,n\}.
\]
For $\ve\in(0,1)$, decompose $B$ as:
\[
\begin{split}
B & = \sum_{i\not\in I}\frac{F^{i\ov{i}}|u_{V_{1}V_{1}i}|^{2}}{\lambda_{1}^{2}}
+2\ve\sum_{i\in I}\frac{F^{i\ov{i}}|u_{V_{1}V_{1}i}|^{2}}{\lambda_{1}^{2}}+(1-2\ve)\sum_{i\in I}\frac{F^{i\ov{i}}|u_{V_{1}V_{1}i}|^{2}}{\lambda_{1}^{2}}\\[2mm]
& =: B_{1}+B_{2}+B_{3}.
\end{split}
\]

\begin{lemma}\label{properties i q}
At $x_{0}$, we have
\begin{enumerate}\setlength{\itemsep}{1mm}
\item For $i\in I$ and $q\notin I$,
\[
F^{q\ov{q}} \leq A^{-4}F^{i\ov{i}} \leq A^{-4}F^{n\ov{n}}, \ \ F^{i\ov{i}} \geq A^{-4n}F^{n\ov{n}};
\]
\item $\sum_{i\in I}\sum_{j}(|u_{ij}|^{2}+|u_{i\ov{j}}|^{2}) \leq C_{A}K$;
\item $\ddbar u\geq-C_{A}K\omega$.
\end{enumerate}
\end{lemma}

\begin{proof}
For (1), it is clear that $q\leq i_{0}<i_{0}+1\leq i$. Then
\[
F^{q\ov{q}} \leq F^{i_{0}\ov{i_{0}}} \leq A^{-4}F^{i_{0}+1\ov{i_{0}+1}} \leq A^{-4}F^{i\ov{i}} \leq A^{-4}F^{n\ov{n}}.
\]
Using $i,i+1,\ldots,n-1\notin S$,
\[
F^{i\ov{i}} > A^{-4}F^{i+1\ov{i+1}} > \cdots > A^{-4(n-i)}F^{n\ov{n}} \geq A^{-4n}F^{n\ov{n}}.
\]

For (2), by \eqref{maximum principle equality} and the Cauchy-Schwarz inequality,
\[
B \leq 3(\xi')^{2} F^{i\ov{i}}|(|\theta|^{2})_{i}|^{2}+3(\eta')^{2} F^{i\ov{i}}|(|\partial u|^{2})_{i}|^{2}+3A^{2}F^{i\ov{i}}|\rho_{i}|^{2}.
\]
Substituting this into \eqref{maximum principle inequality}, dropping the non-negative terms $G_{i}$ ($i=1,2,3$) and using $\xi''=6(\xi')^{2}$, $\eta''=6(\eta')^{2}$,
\begin{equation}\label{properties i q eqn 1}
\frac{1}{8K}\sum_{i,j}F^{i\ov{i}}(|u_{ij}|^{2}+|u_{i\ov{j}}|^{2}) \leq C(A^{2}+1)\mathcal{F}+C_{A}.
\end{equation}
By Maclaurin's inequality,
\begin{equation}\label{F lower bound}
\begin{split}
\mathcal{F} = {} & \sum_{i}F^{i\ov{i}} = \frac{n-m+1}{m}\,\sigma_{m}^{1/m-1}\sigma_{m-1} \\
\geq {} & \frac{n-m+1}{m}\binom{n}{m-1}\binom{n}{m}^{\frac{1}{m}-1} \geq 1.
\end{split}
\end{equation}
Combining this with \eqref{properties i q eqn 1} and using (1),
\[
\frac{F^{n\ov{n}}}{8A^{4n}K}\sum_{i\in I}\sum_{j}(|u_{ij}|^{2}+|u_{i\ov{j}}|^{2}) \leq C_{A}\mathcal{F} \leq \frac{C_{A}}{n}F^{n\ov{n}},
\]
which implies (2).

Using (2) and $n\in I$, we see $u_{n\ov{n}}\geq-C_{A}K$ and obtain (3).
\end{proof}

If $I=\{1,\ldots,n\}$, then Lemma \ref{properties i q} (2) implies the required estimate $\lambda_{1}\leq C_{A}K$. Hence, without loss of generality, we assume that $I\neq\{1,\ldots,n\}$ from now on.

\begin{lemma}\label{B1 B2}
At $x_{0}$, we have
\[
B_{1}+B_{2} \leq \xi''F^{i\ov{i}}|(|\theta|^{2})_{i}|^{2}+\eta''F^{i\ov{i}}|(|\partial u|^{2})_{i}|^{2}+(CA^{-2}+C\ve A^{2})\mathcal{F}.
\]
\end{lemma}

\begin{proof}
By \eqref{maximum principle equality} and the Cauchy-Schwarz inequality,
\[
\begin{split}
B_{1} \leq 3(\xi')^{2}\sum_{i\notin I}F^{i\ov{i}}|(|\theta|^{2})_{i}|^{2}
+3(\eta')^{2}\sum_{i\notin I}F^{i\ov{i}}|(|\partial u|^{2})_{i}|^{2}
+3A^{2}\sum_{i\notin I}F^{i\ov{i}}|\rho_{i}|^{2}
\end{split}
\]
and
\[
\begin{split}
B_{2} \leq 6\ve(\xi')^{2}\sum_{i\in I}F^{i\ov{i}}|(|\theta|^{2})_{i}|^{2}
+6\ve(\eta')^{2}\sum_{i\in I}F^{i\ov{i}}|(|\partial u|^{2})_{i}|^{2}
+6\ve A^{2}\sum_{i\in I}F^{i\ov{i}}|\rho_{i}|^{2}.
\end{split}
\]
Using $\xi''=6(\xi')^{2}$, $\eta''=6(\eta')^{2}$, $\ve\in(0,1)$ and Lemma \ref{properties i q} (1),
\[
\begin{split}
B_{1}+B_{2} \leq {} & \xi''F^{i\ov{i}}|(|\theta|^{2})_{i}|^{2}+\eta''F^{i\ov{i}}|(|\partial u|^{2})_{i}|^{2}
+CA^{2}\sum_{q\notin I}F^{q\ov{q}}+C\ve A^{2}\sum_{i\in I}F^{i\ov{i}}|\rho_{i}|^{2} \\
\leq {} & \xi''F^{i\ov{i}}|(|\theta|^{2})_{i}|^{2}+\eta''F^{i\ov{i}}|(|\partial u|^{2})_{i}|^{2}+CA^{-2}\mathcal{F}+C\ve A^{2}\mathcal{F}.
\end{split}
\]
\end{proof}

\begin{lemma}\label{B3}
Choosing $\ve=A^{-2}$, then at $x_{0}$, we have
\[
B_{3} \le G_{1}+G_{2}+G_{3}+\mathcal{F}.
\]
\end{lemma}

\begin{proof}
Let $J$ be the complex structure of $(\Omega,\omega)$. Define the $(1,0)$-vector field
\[
W_{1} = \frac{1}{\sqrt{2}}(V_{1}-\sqrt{-1}JV_{1})
\]
and write
\[
W_{1} = \sum_{q}\nu_{q}\de_{q}, \ \
JV_{1} = \sum_{\alpha>1}\mu_{\alpha}V_{\alpha},
\]
where we have used that $V_1$ is orthogonal to $JV_1$. At $x_{0}$, since $V_{1}$ and $\de_{q}$ are $g$-unit, then
\[
\sum_{q}|\nu_{q}|^{2} = 1, \quad
\sum_{\alpha>1}\mu_{\alpha}^{2} = 1.
\]
By the above definitions, we compute
\[
\begin{split}
u_{V_{1}V_{1}i} = {} & -\sqrt{-1}u_{V_{1}JV_{1}i}+\sqrt{2}u_{V_{1}\ov{W_{1}}i} \\[3mm]
= {} & -\sqrt{-1}\sum_{\alpha>1}\mu_{\alpha}u_{V_{1}V_{\alpha}i}+\sqrt{2}u_{i\ov{W_{1}}V_{1}}+O(K) \\
= {} & -\sqrt{-1}\sum_{\alpha>1}\mu_{\alpha}u_{V_{1}V_{\alpha}i}+\sqrt{2}\sum_{q\notin I}\ov{\nu_{q}}u_{i\ov{q}V_{1}}
+\sqrt{2}\sum_{q\in I}\ov{\nu_{q}}u_{i\ov{q}V_{1}}+O(K),
\end{split}
\]
where $O(K)$ denotes a term satisfying $|O(K)|\leq CK$ for some uniform constant $C$. For $\gamma>0$, using the Cauchy-Schwarz inequality and assuming without loss of generality that $\lambda_{1}^{2}\geq\ve^{-1}K^{2}$,
\[
\begin{split}
B_{3} \leq {} & (1-\ve)\sum_{i\in I}\frac{F^{i\ov{i}}}{\lambda_{1}^{2}}
\left|-\sqrt{-1}\sum_{\alpha>1}\mu_{\alpha}u_{V_{1}V_{\alpha}i}+\sqrt{2}\sum_{q\notin I}\ov{\nu_{q}}u_{i\ov{q}V_{1}}\right|^{2} \\
& +\frac{C}{\ve}\sum_{i\in I}\sum_{q\notin I}\frac{F^{i\ov{i}}|\ov{\nu_{q}}u_{i\ov{q}V_{1}}|^{2}}{\lambda_{1}^{2}}+\frac{CK^{2}\mathcal{F}}{\ve\lambda_{1}^{2}} \\[2mm]
\leq {} & B_{31}+B_{32}+B_{33}+C\mathcal{F},
\end{split}
\]
where
\[
B_{31} = (1-\ve)\left(1+\frac{1}{\gamma}\right)\sum_{i\in I}\frac{F^{i\ov{i}}}{\lambda_{1}^{2}}\left|\sum_{\alpha>1}\mu_{\alpha}u_{V_{1}V_{\alpha}i}\right|^{2},
\]
\[
B_{32} = (1-\ve)(1+\gamma)\sum_{i\in I}\frac{2F^{i\ov{i}}}{\lambda_{1}^{2}}\left|\sum_{q\notin I}\ov{\nu_{q}}u_{i\ov{q}V_{1}}\right|^{2}, \ \
B_{33} = \frac{C}{\ve}\sum_{i\in I}\sum_{q\notin I}\frac{F^{i\ov{i}}|\nu_{q}u_{i\ov{q}V_{1}}|^{2}}{\lambda_{1}^{2}}.
\]

\bigskip

We first deal the term $B_{33}$. Using Lemma \ref{properties i q} (2) and $I=\{i_{0}+1,\ldots,n\}$,
\[
\sum_{\alpha=2i_{0}+1}^{2n}\sum_{\beta=1}^{2n}|u_{\alpha\beta}| \leq C_{A}K.
\]
Recalling that $V_{1}$ is the $g$-unit eigenvector of $\nabla^{2}u$ corresponding to $\lambda_{1}$, we obtain
\[
|V_{1}^{\alpha}| = \left|\frac{1}{\lambda_{1}}\sum_{\beta}u_{\alpha\beta}V_{1}^{\beta}\right|
\leq \frac{C_{A}K}{\lambda_{1}}, \ \ \text{for $\alpha\geq 2i_{0}+1$}
\]
and so
\[
|\nu_{q}| \leq |V_{1}^{2q-1}|+|V_{1}^{2q}| \leq \frac{C_{A}K}{\lambda_{1}}, \ \ \text{for $q\notin I$}.
\]
It then follows that
\[
B_{33} \leq \frac{C_{A}K^{2}}{\ve\lambda_{1}^{2}}\sum_{i\in I}\sum_{q\notin I}\frac{F^{i\ov{i}}|u_{i\ov{q}V_{1}}|^{2}}{\lambda_{1}^{2}}.
\]
On the other hand, it is clear that
\[
|u_{i\ov{q}V_{1}}| \leq \sum_{\alpha,\beta}|u_{\alpha\beta i}|+CK.
\]
Thus,
\begin{equation}\label{B 33}
\begin{split}
B_{33} \leq {} &\frac{C_{A}K^{2}}{\ve\lambda_{1}^{2}} \sum_{\alpha,\beta}\frac{F^{i\ov{i}}|u_{\alpha\beta i}|^{2}}{\lambda_{1}^{2}}
+\frac{C_{A}K^{4}\mathcal{F}}{\ve\lambda_{1}^{4}} \\
\leq {} & \frac{C_{A}^{2}K^{2}}{\ve\lambda_{1}^{2}}G_{3}+\frac{C_{A}K^{4}\mathcal{F}}{\ve\lambda_{1}^{4}} \\[3mm]
\leq {} & G_{3}+\mathcal{F},\\[1mm]
\end{split}
\end{equation}
provided that $\lambda_{1}^{2}\geq\ve^{-1}C_{A}^{2}K^{2}$ and $\lambda_{1}^{4}\geq\ve^{-1}C_{A}K^{4}$.

\bigskip

We next deal with the terms $B_{31}$ and $B_{32}$. Thanks to \eqref{B 33}, to prove Lemma \ref{B3}, it suffices to show
\begin{equation}\label{B 31 B 32}
B_{31}+B_{32} \leq G_{1}+G_{2}.
\end{equation}
 By the Cauchy-Schwarz inequality and $\sum_{\alpha>1}\mu_{\alpha}^{2}=1$,
\begin{equation}\label{B 31}
\begin{split}
B_{31} \leq {} & (1-\ve)\left(1+\frac{1}{\gamma}\right)
\sum_{i\in I}\frac{F^{i\ov{i}}}{\lambda_{1}^{2}}\left(\sum_{\alpha>1}(\lambda_{1}-\lambda_{\alpha})\mu_{\alpha}^{2}\right)
\left(\sum_{\alpha>1}\frac{|u_{V_{1}V_{\alpha}i}|^{2}}{\lambda_{1}-\lambda_{\alpha}}\right) \\
= {} & \frac{1-\ve}{2\lambda_{1}}\left(1+\frac{1}{\gamma}\right)\left(\lambda_{1}-\sum_{\alpha>1}\lambda_{\alpha}\mu_{\alpha}^{2}\right)G_{1}.
\end{split}
\end{equation}
For $q\notin I$ and $i\in I$, we observe that $u_{q\ov{q}}>u_{i\ov{i}}$. Otherwise, we obtain $u_{q\ov{q}}=u_{i\ov{i}}$ and so $F^{q\ov{q}}=F^{i\ov{i}}$, which contradicts with Lemma \ref{properties i q} (1). Using the Cauchy-Schwarz inequality again,
\[
\begin{split}
B_{32} \leq {} &
(1-\ve)(1+\gamma)\sum_{i\in I}\frac{2F^{i\ov{i}}}{\lambda_{1}^{2}}
\left(\sum_{q\notin I}\frac{(u_{q\ov{q}}-u_{i\ov{i}})|\nu_{q}|^{2}}{F^{i\ov{i}}-F^{q\ov{q}}}\right)
\left(\sum_{q\notin I}\frac{(F^{i\ov{i}}-F^{q\ov{q}})|u_{i\ov{q}V_{1}}|^{2}}{u_{q\ov{q}}-u_{i\ov{i}}}\right).
\end{split}
\]
For $i\in I$, using Lemma \ref{properties i q} (2) and (3),
\[
\begin{split}
& \sum_{q\notin I}\frac{(u_{q\ov{q}}-u_{i\ov{i}})|\nu_{q}|^{2}}{F^{i\ov{i}}-F^{q\ov{q}}}
\leq \sum_{q\notin I}\frac{(u_{q\ov{q}}-u_{i\ov{i}})|\nu_{q}|^{2}}{(1-A^{-4})F^{i\ov{i}}} \\
= {} & \frac{1}{(1-A^{-4})F^{i\ov{i}}}\left(\sum_{q}u_{q\ov{q}}|\nu_{q}|^{2}-\sum_{q\in I}u_{q\ov{q}}|\nu_{q}|^{2}
-u_{i\ov{i}}\sum_{q\notin I}|\nu_{q}|^{2}\right) \\
\leq {} & \frac{1}{(1-A^{-4})F^{i\ov{i}}}\left(u_{W_{1}\ov{W_{1}}}+C_{A}K\right).
\end{split}
\]
Then
\[
B_{32} \leq
\frac{(1-\ve)(1+\gamma)}{(1-A^{-4})\lambda_{1}}
\left(u_{W_{1}\ov{W_{1}}}+C_{A}K\right)
\cdot\frac{2}{\lambda_{1}}\sum_{i\in I}\sum_{q\notin I}\frac{(F^{i\ov{i}}-F^{q\ov{q}})|u_{i\ov{q}V_{1}}|^{2}}{u_{q\ov{q}}-u_{i\ov{i}}}.
\]
Using $F^{i\ov{q},q\ov{i}}=\frac{F^{i\ov{i}}-F^{q\ov{q}}}{u_{i\ov{i}}-u_{q\ov{q}}}$, we have
\[
\frac{2}{\lambda_{1}}\sum_{i\in I}\sum_{q\notin I}\frac{(F^{i\ov{i}}-F^{q\ov{q}})|u_{i\ov{q}V_{1}}|^{2}}{u_{q\ov{q}}-u_{i\ov{i}}}
= -\frac{2}{\lambda_{1}}\sum_{i\in I}\sum_{q\notin I}F^{i\ov{q},q\ov{i}}|u_{i\ov{q}V_{1}}|^{2} \leq G_{2}
\]
and so
\begin{equation}\label{B 32}
\begin{split}
B_{32} \leq {} & \frac{(1-\ve)(1+\gamma)}{(1-A^{-4})\lambda_{1}}\left(u_{W_{1}\ov{W_{1}}}+C_{A}K\right)G_{2} \\
= {} & \frac{1+\gamma}{(1+\ve)\lambda_{1}}\left(u_{W_{1}\ov{W_{1}}}+C_{A}K\right)G_{2},
\end{split}
\end{equation}
where we used $\ve=A^{-2}$ in the second line. By Lemma \ref{properties i q} (3), increasing $C_{A}$ if needed, we assume that $u_{W_{1}\ov{W_{1}}}+C_{A}K$ is positive. To show \eqref{B 31 B 32}, we split the argument into two cases.

\bigskip
\noindent
{\bf Case 1.} $\frac{1}{2}(\lambda_{1}+\sum_{\alpha>1}\lambda_{\alpha}\mu_{\alpha}^{2})>\frac{1}{1+\ve}(u_{W_{1}\ov{W_{1}}}+C_{A}K)>0$.
\bigskip

Using \eqref{B 31} and \eqref{B 32},
\[
B_{31}+B_{32}
\leq \frac{1-\ve}{2\lambda_{1}}\left(1+\frac{1}{\gamma}\right)\left(\lambda_{1}-\sum_{\alpha>1}\lambda_{\alpha}\mu_{\alpha}^{2}\right)G_{1}
+\frac{1+\gamma}{2\lambda_{1}}\left(\lambda_{1}+\sum_{\alpha>1}\lambda_{\alpha}\mu_{\alpha}^{2}\right)G_{2}.
\]
Since $\lambda_{1}>\lambda_{2}$, after choosing
\[
\gamma = \frac{\lambda_{1}-\sum_{\alpha>1}\lambda_{\alpha}\mu_{\alpha}^{2}}{\lambda_{1}+\sum_{\alpha>1}\lambda_{\alpha}\mu_{\alpha}^{2}} > 0,
\]
we obtain \eqref{B 31 B 32}.

\bigskip
\noindent
{\bf Case 2.} $\frac{1}{2}(\lambda_{1}+\sum_{\alpha>1}\lambda_{\alpha}\mu_{\alpha}^{2})\leq\frac{1}{1+\ve}(u_{W_{1}\ov{W_{1}}}+C_{A}K)$.
\bigskip

By the definitions of $W_{1}$ and $\mu_{\alpha}$,
\begin{equation}\label{case 2 eqn 1}
u_{W_{1}\ov{W_{1}}} = \frac{1}{2}(u_{V_{1}V_{1}}+u_{JV_{1}JV_{1}}) = \frac{1}{2}\left(\lambda_{1}+\sum_{\alpha>1}\lambda_{\alpha}\mu_{\alpha}^{2}\right).
\end{equation}
Combining this with the assumption of Case 2,
\[
u_{W_{1}\ov{W_{1}}} \leq \frac{1}{1+\ve}(u_{W_{1}\ov{W_{1}}}+C_{A}K).
\]
Then $u_{W_{1}\ov{W_{1}}}\leq\ve^{-1}C_{A}K$ and so
\begin{equation}\label{case 2 eqn 2}
u_{W_{1}\ov{W_{1}}}+C_{A}K \leq \frac{C_{A}K}{\ve}.
\end{equation}
Using \eqref{case 2 eqn 1} and Lemma \ref{properties i q} (2),
\[
\lambda_{1}+\sum_{\alpha>1}\lambda_{\alpha}\mu_{\alpha}^{2} = 2u_{W_{1}\ov{W_{1}}} \geq -C_{A}K,
\]
which implies
\begin{equation}\label{case 2 eqn 3}
\lambda_{1}-\sum_{\alpha>1}\lambda_{\alpha}\mu_{\alpha}^{2}
\leq 2\lambda_{1}+C_{A}K \leq 2(1+\ve^{2})\lambda_{1},
\end{equation}
provided that $\lambda_{1}\geq\ve^{-1}C_{A}K$. Hence, combining \eqref{B 31}, \eqref{B 32}, \eqref{case 2 eqn 2}, \eqref{case 2 eqn 3} and choosing $\gamma=\ve^{-2}$,
\[
\begin{split}
B_{31}+B_{32} \leq {} & (1-\ve)(1+\ve^{2})^{2}G_{1}+\frac{C_{A}K}{\ve^{3}\lambda_{1}}G_{2}
\leq G_{1}+\frac{C_{A}K}{\ve^{3}\lambda_{1}}G_{2}.
\end{split}
\]
This implies \eqref{B 31 B 32} as long as $\lambda_{1}\geq\ve^{-3}C_{A}K$.
\end{proof}

\subsubsection{Completion of the proof}
Substituting Lemma \ref{B1 B2} and \ref{B3} into \eqref{maximum principle inequality}, we obtain
\[
\frac{1}{8K}\sum_{i,j}F^{i\ov{i}}(|u_{ij}|^{2}+|u_{i\ov{j}}|^{2})+(A-C_{1})\mathcal{F} \leq C_{A}
\]
for some uniform constant $C_{1}$. Choosing $A=C_{1}+1$,
\[
\frac{1}{8K}\sum_{i,j}F^{i\ov{i}}(|u_{ij}|^{2}+|u_{i\ov{j}}|^{2})+\mathcal{F} \leq C.
\]
By \cite[Lemma 2.2 (2)]{HMW10}, we have $F^{i\ov{i}}\geq C^{-1}$ and so
\[
\sum_{i,j}F^{i\ov{i}}(|u_{ij}|^{2}+|u_{i\ov{j}}|^{2}) \leq CK,
\]
which implies $\lambda_{1}\leq CK$, as required.
\end{proof}

\subsection{Boundary Hessian estimate}
\begin{proposition}\label{boundary Hessian estimate}
Under the same assumptions of Theorem \ref{a priori estimates}, there exists a uniform constant $C$ such that
\begin{equation}\label{boundary second estimate}
\sup_{\de\Omega}|\nabla^{2}u| \leq C\sup_{\Omega}|\de u|^{2}+C\sup_{\Omega}|\de v|^{2}+C.
\end{equation}
\end{proposition}

\begin{proof}
We mostly follow Collins-Picard \cite{CP22}, with the following differences. First, as already remarked in the introduction, we by-pass their complicated boundary normal-normal derivative estimate by using the ``supersolution" $w$ and $m$-pseudoconvexity of the boundary. Secondly, we require the constant in \eqref{boundary second estimate} to be independent of $\inf_{\Omega}\psi$, which requires some further modifications throughout the argument.

We divide the proof into three steps.

\bigskip
\noindent
{\bf Step 1:} Setup and tangent-tangent derivatives.
\bigskip

We start with some setup. Fix a point $p\in \p\Omega$, and suppose that $z = (z^1, \ldots, z^n)$ is a coordinate system centered at $p$, valid on some small ball $B$. Let:
\[
T^{1,0}_{p} \p \Omega := \{ V\in T^{1,0}_p \Omega\ | \ V(\rho) = 0\};
\]
then, after a coordinate change, we can assume that $\omega_{i\ov{j}}(0) = \delta_{i\ov{j}}$,
\[
T^{1,0}_{p} \p \Omega  = \mathrm{Span}\left\{ \frac{\p}{\p z^1}, \ldots, \frac{\p}{\p z^{n-1}}\right\},
\]
and that $x^n > 0$ in $\Omega\cap B$, where we write $z^i = x^i + \sqrt{-1} y^i$, $i\in \{1, \ldots, n\}$. It follows that, after rescaling $\rho$ by a positive constant if necessary, we have:
\[
\rho(z) = -x_n + O(|z|^2)
\]
in $B$. Define $t = (t^1, \ldots, t^{2n-1})$ by $t^\alpha = y^\alpha$ for $1 \leq \alpha \leq n$ and $t^{\alpha} = x^{\alpha - n}$ for $n+1\leq \alpha \leq 2n-1$. From \eqref{un u u w condition}, we obtain
\[
|\p_{t^\alpha}\p_{t^\beta} u(0)| = |-\p_{x^n} u(0) \p_{t^\alpha}\p_{t^\beta} \rho(0)| \leq C, \ \ \text{for $1\leq \alpha, \beta\leq 2n-1$}.
\]

\bigskip
\noindent
{\bf Step 2:} Tangent-normal derivatives.
\bigskip

The main task is to deal with the tangent-normal derivatives, which we bound using a barrier argument. We follow \cite[Section 4]{CP22} to construct the barrier function. Define
\[
K := \sup_{\Omega}|\de u|^{2}+\sup_{\Omega}|\de v|^{2}+1, \quad
\mathcal{G} := \sigma_{m}^{p\ov{q}}\omega_{p\ov{q}}.
\]
Let $\{e_a\}_{a=1}^n$ be an orthonormal frame for $T^{1,0}\Omega$ such that the $\{e_{a}\}_{a=1}^{n}$ are tangential to the level sets of $\rho$ and $e_{a}(0)=\de_{a}$ for $1\leq a\leq n-1$ (see the start of \cite[Section 4.5]{CP22} for the construction). Write $e_{a}=e_{a}^{i}\de_{i}$ and set:
\[
U = A K^{1/2} V + BK^{1/2}|z|^2 - \frac{1}{K^{1/2}}\sum_{i=1}^n (\p_{y^i}(u-\ul{u}))^2 - \frac{1}{K^{1/2}}\sum_{a=1}^{n-1}\abs{\nabla_a (u-\ul{u})}^2,
\]
where we define the function $V$ by:
\[
V := u - \ul{u} + c_0 d - N_0 d^2,
\]
with $d$ the $\omega$-distance function to $\p \Omega$ and $A,B,c_0,N_0>0$ uniform constants to be chosen later.

Fix $\alpha\in \{1,\ldots, 2n-1\}$ and consider the tangential operators:
\[
T_\alpha := \frac{\p}{\p t^\alpha} - \frac{\rho_{t^\alpha}}{\rho_{x^n}}\frac{\p}{\p x^n}.
\]
It is clear that $T_{\alpha}(u-\un{u})=0$ on $\de\Omega\cap B_{\delta}$. We claim that
\begin{equation}\label{claim}
\begin{cases}
\ U(0) = 0 & \\[1mm]
\ U \geq \abs{T_\alpha(u - \ul{u})} & \mbox{on $\p(\Omega\cap B_{\delta})$}\\[1mm]
\ \sigma^{p\ov{q}}_k \p_p\p_{\ov{q}}U \leq -\abs{\sigma^{p\ov{q}}\p_p\p_{\ov{q}} T_\alpha (u-\ul{u})} & \mbox{in $\Omega\cap B_{\delta}$}.
\end{cases}
\end{equation}
Accepting the claim, we show how to conclude the argument. First, by the maximum principle, the claim implies that
\[
U\pm T_{\alpha}(u-\un{u}) \geq 0 \ \ \text{in $\Omega\cap B_{\delta}$}.
\]
Since $U(0) = T_\alpha(u - \ul{u})(0) = 0$, we have that
\[
\p_{x^n} (U \pm T_\alpha(u - \ul{u}))(0) \geq 0,
\]
and hence
\[
\abs{\p_{x^n}\p_{t^\alpha} u(0)}
\leq \left|AK^{1/2}\de_{x^{n}}V(0)\right|+\left|\left(\de_{x^{n}}\frac{\rho_{t^{\alpha}}}{\rho_{x^{n}}}\right)(0)\cdot\de_{x^{n}}(u-\un{u})(0)\right|+\left|\p_{x^n}\p_{t^\alpha}\un{u}(0)\right|.
\]
Combining this with $\sup_{\de\Omega}|\de u|\leq C$, we obtain the tangent-normal derivative estimate:
\[
\abs{\p_{x^n}\p_{t^\alpha} u(0)} \leq CK^{1/2}.
\]
\medskip

We are left to verify the properties of $U$ listed in \eqref{claim}. The first property is immediate. For the second, we separate $\p(\Omega\cap B_{\delta})$ into the two pieces $\p\Omega\cap B_{\delta}$ and $\Omega\cap \p B_{\delta}$. On the first piece, we have:
\[
T_\alpha(u - \ul{u}) = \nabla_a(u-\ul{u}) = 0 \text{ and }(\p_{y^i}(u-\ul{u}))^2\leq C|z|^2,
\]
giving:
\[
U - \abs{T_\alpha(u-\ul{u})} \geq AK^{1/2}V + BK^{1/2}\abs{z}^2 - C|z|^2 \geq AK^{1/2}V,
\]
for $B\gg1$. On the second piece, we have $|z| = \delta$ (which we choose later), giving:
\[
U - \abs{T_\alpha(u-\ul{u})} \geq AK^{1/2}V + BK^{1/2}\delta^2 - CK^{1/2} \geq AK^{1/2}V,
\]
for $B\gg1$. Then the second property in \eqref{claim} follows from the above and $V\geq0$ (see Lemma \ref{V Lemma} below).
\begin{lemma}[Lemma 4.2 of \cite{CP22}]\label{V Lemma}
There exist uniform constants $\delta, c_0, N_{0}, \tau > 0$ such that
\[
 \sigma_{m}^{p\ov{q}}\p_p\p_{\ov{q}} V \leq -\tau(\psi^{m-1}+\mathcal{G}) \ \text{and} \ V \geq 0 \ \text{in $\Omega\cap B_{\delta}$}.
\]
\end{lemma}

\begin{proof}
Recall $\sigma_{m}=\psi^{m}$ and $\mathcal{G}=\sigma_{m}^{p\ov{q}}\omega_{p\ov{q}}$.
The first three equations in \cite[page 1654]{CP22} imply
\[
\begin{split}
\sigma_{m}^{p\ov{q}}\de_{p}\de_{\ov{q}}(u-\un{u})
\leq {} & -\tau\mathcal{G}+\sigma_{m}\big(\log\sigma_{m}+C(\un{u})\big) \\
\leq {} & -\tau\mathcal{G}+\psi^{m-1}\big(\psi\log\psi^{m}+C(\un{u})\psi\big) \\
\leq {} & -\tau\mathcal{G}+C(\un{u},\psi)\psi^{m-1},
\end{split}
\]
where we used the fact that $x\log x^{m}\leq C(m,D)$ for $x\in(0,D]$. Then \cite[(4.13)]{CP22} becomes
\[
\sigma_{m}^{p\ov{q}}\de_{p}\de_{\ov{q}}V \leq -\frac{3\tau}{4}\mathcal{G}-\frac{N_{0}}{2}\sigma_{m}^{1\ov{1}}+C\psi^{m-1}.
\]
By the first three equations in \cite[page 1655]{CP22},
\[
-\frac{\tau}{2}\mathcal{G}-\frac{N_{0}}{2}\sigma_{m}^{1\ov{1}}
\leq -m\sigma_{m}^{\frac{m-1}{m}}\left(\frac{N_{0}\tau^{m-1}}{2^{m}}\right)^{\frac{1}{m}}
= -\frac{m\tau^{\frac{m-1}{m}}N_{0}^{\frac{1}{m}}}{2}\psi^{m-1}.
\]
Choosing $N\gg1$ and renaming $\tau$, we obtain $\sigma_{m}^{p\ov{q}}\p_p\p_{\ov{q}} V \leq -\tau(\psi^{m-1}+\mathcal{G})$. The part $V\geq0$ follows from the same argument of \cite[Lemma 4.2]{CP22}.
\end{proof}

We now show that we can choose $A\gg1$ such that the third property in \eqref{claim} holds. We use $\mathcal{E}$ to denote a term satisfying
\[
|\mathcal{E}| \leq CK^{1/2}\mathcal{G}+C\sum_{i}\sigma_{m-1}(\Lambda|i)|\Lambda_i|+C\psi^{m-1}.
\]

\begin{lemma}[(4.16), (4.26), (4.31) of \cite{CP22}]\label{claim inequalities}
The following inequalities hold:
\[
\bullet\ \abs{\sigma_{m}^{p\ov{q}}\p_p\p_{\ov{q}} T_{\alpha}(u - \ul{u})} \leq \frac{1}{K^{1/2}}\sigma_{m}^{p\ov{q}}\p_p\p_{y^n}(u - \ul{u})\cdot \p_{\ov{q}}\p_{y^n}(u-\ul{u})+\mathcal{E},
\]
\[
\bullet -\frac{1}{K^{1/2}}\sigma_{m}^{p\ov{q}}\p_p\p_{\ov{q}} (\p_{y^i}(u-\ul{u}))^2
\leq -\frac{2}{K^{1/2}}\sigma_{m}^{p\ov{q}}\p_p\p_{y^i}(u - \ul{u})\cdot \p_{\ov{q}}\p_{y^i}(u-\ul{u}) +\mathcal{E},
\]
\[
\begin{split}
\bullet-\frac{1}{K^{1/2}}\sigma_{m}^{p\ov{q}}\p_p\p_{\ov{q}}\sum_{a=1}^{n-1} & \abs{\nabla_a(u-\ul{u})}^2 \leq -\frac{1}{2nK^{1/2}}\sum_{i\not=r}\sigma_{m-1}(\Lambda|i)\Lambda_i^2 \\
&+ \frac{1}{K^{1/2}}\sum_{i=1}^n \sigma_{m}^{p\ov{q}}\p_p\p_{y^i}(u - \ul{u})\cdot \p_{\ov{q}}\p_{y^i}(u-\ul{u})+\mathcal{E},
\end{split}
\]
where $r$ is some chosen index such that $1/n \leq |e^r_n|^2\leq 1$.
\end{lemma}

\begin{proof}
It suffices to show that for $1\leq\beta\leq 2n$,
\begin{equation}\label{third order term}
\sigma_{m}^{p\ov{q}}\de_{p}\de_{\ov{q}}\de_{\beta}u = \mathcal{E}.
\end{equation}
Given this, the lemma follows from the arguments of \cite[(4.16), (4.26), (4.31)]{CP22}. By the definitions of $t^{\beta}$ and $z^{i}$,
\[
\frac{\de}{\de t^{\beta}} =
\begin{cases}
\ \frac{\de}{\de z^{\beta-n}}+\frac{\de}{\de \ov{z}^{\beta-n}} & \mbox{if $1\leq\beta\leq n$}, \\[1mm]
\ \frac{1}{\sqrt{-1}}\left(\frac{\de}{\de z^{\beta-n}}-\frac{\de}{\de \ov{z}^{\beta-n}}\right) & \mbox{if $n+1\leq\beta\leq 2n$}.
\end{cases}
\]
Since both cases are similar, we only deal with the first case. Applying $\nabla_{\beta}$ to the equation $\sigma_{m}=\psi^{m}$,
\[
\sigma_{m}^{p\ov{q}}\nabla_{\beta}\nabla_{p}\nabla_{\ov{q}}u = \nabla_{\beta}\psi^{m}.
\]
Writing $i=\beta-n$ and using \cite[(4.18)]{CP22},
\begin{equation}\label{third order term eqn 1}
\begin{split}
\sigma_{m}^{p\ov{q}}\de_{p}\de_{\ov{q}}\de_{\beta}u
= {} & \sigma_{m}^{p\ov{q}}\left(\nabla_{\beta}\nabla_{p}\nabla_{\ov{q}}u+\Gamma_{ip}^{r}u_{r\ov{q}}+\Gamma_{\ov{i}\ov{q}}^{\ov{r}}u_{p\ov{r}}\right) \\
= {} & m\psi^{m-1}(\de_{\beta}\psi+\de_{v}\psi\cdot\de_{\beta}v)+\sigma_{m}^{p\ov{q}}\Gamma_{ip}^{r}u_{r\ov{q}}+\sigma_{m}^{p\ov{q}}\Gamma_{\ov{i}\ov{q}}^{\ov{r}}u_{p\ov{r}}.
\end{split}
\end{equation}
Maclaurin's inequality (see \eqref{F lower bound}) shows that
\[
\mathcal{G} = \sigma_{m}^{p\ov{q}}\omega_{p\ov{q}} = (n-m+1)\sigma_{m-1}
\geq m\sigma_{m}^{\frac{m-1}{m}} = m\psi^{m-1},
\]
which implies
\begin{equation}\label{third order term eqn 2}
|m\psi^{m-1}(\de_{\beta}\psi+\de_{v}\psi\cdot\de_{\beta}v)| \leq CK^{1/2}\mathcal{G}.
\end{equation}
Using \cite[(4.21)]{CP22},
\begin{equation}\label{third order term eqn 3}
\sigma_{m}^{p\ov{q}}\Gamma_{ip}^{r}u_{r\ov{q}}+\sigma_{m}^{p\ov{q}}\Gamma_{\ov{i}\ov{q}}^{\ov{r}}u_{p\ov{r}} = \mathcal{E}.
\end{equation}
Substituting \eqref{third order term eqn 2} and \eqref{third order term eqn 3} into \eqref{third order term eqn 1}, we obtain \eqref{third order term}.
\end{proof}

\begin{lemma}[Corollary 2.21 of \cite{Guan14}]\label{index r Lemma}
For any $\ve>0$ and index $r$,
\[
\sum_{i}\sigma_{m-1}(\Lambda|i)|\Lambda_i|
\leq \ve\sum_{i\not=r} \sigma_{m-1}(\Lambda|i)\Lambda_i^2
+\frac{C\mathcal{G}}{\ve}+C\psi^{m-1}.
\]
\end{lemma}

\begin{proof}
Observe that the last equation in \cite[page 1650]{CP22} can be refined as follows:
\begin{equation}\label{index r Lemma eqn}
\sum_{i}\sigma_{m}^{i\ov{i}}(\Lambda_{i}-1) \leq \sigma_{m}(\log\sigma_{m}+C)
= \psi^{m-1}(\psi\log\psi^{m}+C\psi) \leq C\psi^{m-1},
\end{equation}
where we again used the fact that $x\log x^{m}\leq C(m,D)$ for $x\in(0,D]$. Then Lemma \ref{index r Lemma} follows from \eqref{index r Lemma eqn} and the same argument of \cite[Lemma 3.4]{CP22}.
\end{proof}

By Lemma \ref{V Lemma} and \ref{claim inequalities},
\[
\begin{split}
& \sigma^{p\ov{q}}\p_p\p_{\ov{q}}(U \pm T_\alpha(u - \ul{u})) \\[2mm]
\leq {} & -A\tau K^{1/2}(\psi^{m-1}+\mathcal{G}) + B K^{1/2}\sum_p \sigma_{m}^{p\ov{p}} \\[-2mm]
& - \frac{1}{2nK^{1/2}}\sum_{i\not=r} \sigma_{m-1}(\Lambda|i)\Lambda_i^2
 -\frac{1}{K^{1/2}}\sum_{i=1}^{n-1}\sigma^{p\ov{q}}\p_p \p_{y^i}(u - \ul{u})\, \p_{\ov{q}}\p_{y^i}(u-\ul{u}) \\[1.5mm]
& +CK^{1/2}\mathcal{G}+C\sum_{i}\sigma_{m-1}(\Lambda|i)|\Lambda_i|+C\psi^{m-1}.
\end{split}
\]
Choosing $A\gg1$,
\[
\begin{split}
& \sigma^{p\ov{q}}\p_p\p_{\ov{q}}(U \pm T_\alpha(u - \ul{u})) \\
\leq {} & -\frac{A\tau K^{1/2}}{2}(\psi^{m-1}+\mathcal{G}) - \frac{1}{2nK^{1/2}}\sum_{i\not=r} \sigma_{m-1}(\Lambda|i)\Lambda_i^2
+C\sum_{i}\sigma_{m-1}(\Lambda|i)|\Lambda_i|.
\end{split}
\]
Using Lemma \ref{index r Lemma} with $\ve=\frac{1}{2nCK^{1/2}}$,
\[
\sigma^{p\ov{q}}\p_p\p_{\ov{q}}(U \pm T_\alpha(u - \ul{u})) \\
\leq -\frac{A\tau K^{1/2}}{2}(\psi^{m-1}+\mathcal{G})+C\psi^{m-1}+CK^{1/2}\mathcal{G}.
\]
Increasing $A$ if needed, we obtain the third property in \eqref{claim}.

\bigskip
\noindent
{\bf Step 3}: The normal-normal derivative.
\bigskip

We are left to show that the normal-normal derivative of $u$ is bounded on $\p\Omega$. By \eqref{un u u w condition} and the Hopf lemma, we obtain
\[
-u_{x_{n}}(0) \geq -w_{x_{n}}(0) \geq C^{-1}
\]
and
\[
u_{i\ov{j}}(0) = -u_{x_{n}}(0)\rho_{i\ov{j}}(0), \ \ \text{for $1\leq i,j\leq n-1$}.
\]
Since $\ddbar\rho\in\Gamma_{m}$, then $[\rho_{i\ov{j}}(0)]_{1\leq i,j\leq n-1}\in\Gamma_{m-1}$. It then follows that
\[
\sigma_{m-1}\left([u_{i\ov{j}}(0)]_{1\leq i,j\leq n-1}\right)
= \left[-u_{x_{n}}(0)\right]^{m-1}\sigma_{m-1}\left([\rho_{i\ov{j}}(0)]_{1\leq i,j\leq n-1}\right) \geq C^{-1}.
\]
The tangent-tangent derivative estimate and tangent-normal derivative estimate imply:
\[
|u_{i\ov{j}}(0)| \leq C, \ \ |u_{i\ov{n}}(0)|+|u_{n\ov{j}}(0)| \leq C\sqrt{K}, \ \ \text{for $1\leq i,j\leq n-1$}.
\]
Since $\sigma_{m}$ is equal to the sum of all $m$-th principal minors, then
\[
\sigma_{m}(u) = u_{n\ov{n}}\sigma_{m-1}\left([u_{i\ov{j}}(0)]_{1\leq i,j\leq n-1}\right)+O(K),
\]
where $O(K)$ denotes a term satisfying $|O(K)|\leq CK$. Thus we obtain $|u_{n\ov{n}}|\leq CK$, which implies our normal-normal derivative estimate.
\end{proof}

\section{Existence theorem}
\label{Existence theorem}
In this section, we solve the following Dirichlet problem for the complex Hessian equation when the right-hand side is decreasing in $u$, by using an iteration argument:
\begin{equation}\label{Dirichlet problem}
\begin{cases}
(\ddbar u)^{m}\wedge\omega^{n-m} = \psi^{m}(z,u)\omega^{n} & \text{in $\Omega$}, \\
u = 0 & \text{on $\de\Omega$}.
\end{cases}
\end{equation}

\begin{theorem}\label{existence under sub-super}
Let $\psi$ be a smooth (strictly) positive function on $\ov{\Omega}\times(-\infty,0]$ such that $\psi$ is decreasing in the second variable. Suppose that there exist a subsolution $\un{u}\in C^\infty(\ov{\Omega})$ and a supersolution $\ov{u} \in C^0(\ov{\Omega})$ such that $\un{u}<\ov{u}<0$ in $\Omega$. Then there exists $u\in\mSH(\Omega)\cap C^{\infty}(\ov\Omega)$ such that
\begin{equation}\label{existence under sub-super eqn}
\begin{cases}
(\ddbar u)^{m}\wedge\omega^{n-m} = \psi^{m}(z,u)\omega^{n} & \text{in $\Omega$}, \\
u = 0 & \text{on $\de\Omega$},
\end{cases}
\end{equation}
and:
\begin{equation}\label{existence under sub-super estimate}
\|u\|_{C^{2}(\ov{\Omega})} \leq C,
\end{equation}
where $C$ is a constant depending only on $(\ov\Omega,\omega)$, $\|\psi\|_{C^{2}}$, $\|\un{u}\|_{C^{2}}$, and the choice of a non-zero function $h\in C^\infty(\ov{\Omega})$ with $0 \leq h \leq \psi(z,\ov{u})$.
\end{theorem}

The role of the function $h$ is to smooth out the non-zero function $\psi(z, \ov{u})$, so the estimates are independent of the $C^2$-norm of $\ov{u}$.

\begin{proof}
We define a sequence of functions $\{u_{j}\}_{j=0}^{\infty}$ inductively. Set $u_{0}=\un{u}$. By Corollary \ref{CP Dirichlet corollary}, there exists $u_{j}\in\mSH(\Omega)\cap C^{\infty}(\ov\Omega)$ solving
\[
\begin{cases}
(\ddbar u_{j})^{m}\wedge\omega^{n-m} = \psi^{m}(z ,u_{j-1})\omega^{n} & \text{in $\Omega$}, \\
u_{j} = 0 & \text{on $\de\Omega$}.
\end{cases}
\]
Since $\psi$ is decreasing in the second variable, induction and the domination principle (see \cite[Subsection 3.5: Step 1]{BZ23a}) show that $\un{u}\leq u_{j-1}\leq u_{j} \leq \ov{u}<0$.

We now smooth out $\ov{u}$. Choose some non-zero function $h\in C^\infty(\ov{\Omega})$ such that:
\[
0 \leq h\leq \psi(z, \ov{u}).
\]
Using $u_{j}\leq\ov{u}<0$, we have
\[
(\ddbar u_{j})^{m}\wedge\omega^{n-m} \geq \psi^{m}(z ,\ov{u})\omega^{n} \geq h^{m} \omega^{n}.
\]
By Maclaurin's inequality,
\[
\Delta_{\omega}u_{j} \geq n\left(\frac{(\ddbar u_{j})^{m}\wedge\omega^{n-m}}{\omega^{n}}\right)^{\frac{1}{m}}
\geq nh \ \text{in $\Omega$}.
\]
We then let $w\in\SH_{1}(\Omega)\cap C^{\infty}(\ov{\Omega})$ be the non-zero solution of the Dirichlet problem
\[
\begin{cases}
\Delta_{\omega}w = nh & \text{in $\Omega$}, \\
w = 0 & \text{on $\de\Omega$};
\end{cases}
\]
from the comparison principle, it follows that for any $j$,
\[
\un{u} \leq u_{j} \leq w < 0 \ \text{in \text{$\Omega$}}, \quad
\un{u} = u_{j} = w = 0 \ \text{on \text{$\de\Omega$}}.
\]

\smallskip

We may now apply Theorem \ref {a priori estimates}. Define
\[
L_{j} = \sup_{\Omega}|\nabla^{2} u_{j}|, \quad\text{ and }\quad K_{j} = \sup_{\Omega}|\de u_{j}|^{2}+1.
\]
Theorem \ref{a priori estimates} shows that:
\[
\|u_{j}\|_{C^{0}} \leq C_{0}, \quad L_{j} \leq \frac{L_{j-1}}{2}+C_{0}K_{j}+C_{0}K_{j-1},
\]
so that
\begin{equation}\label{thingy}
L_{j} \leq \frac{L_{0}}{2^{j}}+3C_{0}\sum_{i=0}^{j}\frac{K_{j-i}}{2^{i}}
\leq L_{0}+3C_{0}\max_{0\leq\beta\leq j}K_{\beta}.
\end{equation}
Fix some $j \geq 1$, and choose $1 \leq \alpha_j \leq j$ such that:
\[
K_{\alpha_{j}} = \max_{0\leq \beta \leq j}K_{\beta}.
\]
Then
\[
\|u_{\alpha_{j}}\|_{C^{0}} \leq C_{0}, \quad
L_{\alpha_{j}} \leq L_{0} + 3C_{0}K_{\alpha_{j}}.
\]
Applying the blow-up argument of \cite[Section 6]{CP22}, we obtain
\[
K_{\alpha_{j}} \leq C,
\]
for some $C$ independent of $j$. Then by \eqref{thingy}, we have $L_\alpha \leq C$ for all $1 \leq \alpha \leq j$ for $C$ independent of $j$, which gives:
\[
\sup_{\Omega}|\nabla^{2}u_{\alpha}|+\sup_{\Omega}|\de u_{\alpha}|^{2} \leq C,
\]
for all $\alpha \geq 1$. By standard elliptic theory, we conclude that $u = \lim_{\alpha\rightarrow\infty} u_\alpha$ is the required solution of \eqref{existence under sub-super eqn}.
\end{proof}

We can use a perturbation argument to deal with the case when $\psi$ is zero on the boundary:

\begin{theorem}\label{existence under sub-super degenerate}
Let $\psi$ be a smooth non-negative function on $\ov{\Omega}\times(-\infty,0]$ such that $\psi$ is (strictly) positive on $\Omega\times(-\infty,0)$ and decreasing in the second variable. Suppose that there exist a strict subsolution $\un{u}\in C^\infty(\ov{\Omega})$ and a supersolution $\ov{u} \in C^0(\ov{\Omega})$ such that $\un{u}<\ov{u}<0$ in $\Omega$. Then there exists $u\in\mSH(\Omega)\cap C^{\infty}(\Omega)\cap C^{1,1}(\ov\Omega)$ such that
\[
\begin{cases}
(\ddbar u)^{m}\wedge\omega^{n-m} = \psi^{m}(z,u)\omega^{n} & \text{in $\Omega$}, \\
u = 0 & \text{on $\de\Omega$}.
\end{cases}
\]
\end{theorem}

\begin{proof}
Let $\ve_{0}$ be the constant in the definition of strict subsolution (see Definition \ref{definitions} (2)). For $\ve\in(0,\ve_{0})$, by Theorem \ref{existence under sub-super}, there exists $u_{\ve}\in\mSH(\Omega)\cap C^{\infty}(\ov\Omega)$ such that
\[
\begin{cases}
(\ddbar u_{\ve})^{m}\wedge\omega^{n-m} = (\psi(z,u_{\ve})+\ve)^m\omega^{n} & \text{in $\Omega$}, \\
u_{\ve} = 0 & \text{on $\de\Omega$},
\end{cases}
\]
and
\begin{equation}\label{u ve C 2}
\|u_{\ve}\|_{C^{2}(\ov{\Omega})} \leq C,
\end{equation}
where $C$ is a constant depending only on $(\ov\Omega,\omega)$, $\|\psi\|_{C^{2}}$, $\|\un{u}\|_{C^{2}}$, and $h$, which we can choose independent of $\e$. Since $\psi$ is (strictly) positive in $\Omega\times(-\infty,0)$, for any $\Omega'\Subset\Omega$, we can find some $c_{\Omega'} > 0$, which is independent of $\e$, and such that
\[
(\psi(z,u_{\ve})+\ve)^m \geq \psi^{m}(z,u_{\ve}) \geq c_{\Omega'} > 0;
\]
by standard elliptic theory, we obtain uniform higher order estimates for $u_{\ve}$ on $\Omega'$. Combining these with \eqref{u ve C 2} and an approximation argument, we conclude.
\end{proof}

\section{Existence and Uniqueness of the Eigenfunction}
\label{Existence and Uniqueness of the Eigenfunction}

In this section we will prove out main results, Theorems \ref{Theorem 1} and \ref{Theorem 2}. We start with:
\begin{theorem}[Theorem \ref{Theorem 1}]\label{Theorem 1 in Section 5}
Suppose that $\Omega$ is a strongly $m$-pseudoconvex manifold and $f\in C^\infty(\ov{\Omega})$ is (strictly) positive. Then there exists a unique $\lambda_1 > 0$ and a unique $u_1\in \mSH(\Omega)\cap C^\infty(\Omega)\cap C^{1,1}(\ov{\Omega})$ solving:
\begin{equation}\label{CHE eigenvalue Section}
\begin{cases}
(\ddbar u_1)^m\wedge \omega^{n-m} = (-\lambda_1 u_1)^m f^m\omega^n& \text{ in }\Omega,\\
u_1 = 0 & \text{ on }\p\Omega,\\
\inf_\Omega u_1 = 1.
\end{cases}
\end{equation}
\end{theorem}

\begin{proof}
Following Lions \cite{Lions85} and Wang \cite{Wang94}, for $\lambda\geq0$, we consider the Dirichlet problem:
\begin{equation}\label{continuity path}
\begin{cases}
(\ddbar u)^{m}\wedge\omega^{n-m} =(1-\lambda u)^m f^m\omega^{n} & \text{in $\Omega$}, \\
u = 0 & \text{on $\de\Omega$}.
\end{cases}
\end{equation}
Define:
\[
I = \{\lambda\geq0~|~\text{there is $u_{\lambda}\in\mSH({\Omega})\cap C^{\infty}(\ov\Omega)$ solving \eqref{continuity path}} \},
\]
and
\[
\lambda_1 = \sup_{\lambda\in I}\lambda.
\]
We split the argument into several steps.

\bigskip
\noindent
{\bf Step 1.} $\lambda_{1}$ is well-defined with $\lambda_{1}\in(0,\infty)$.
\bigskip

By Corollary \ref{CP Dirichlet corollary}, when $\lambda=0$, there exists $u_{0}\in\mSH({\Omega})\cap C^{\infty}(\ov\Omega)$ solving
\[
\begin{cases}
(\ddbar u_{0})^{m}\wedge\omega^{n-m} =  f^m\omega^{n} & \text{in $\Omega$}, \\
u_{0} = 0 & \text{on $\de\Omega$}.
\end{cases}
\]
Then $I$ is not empty and so $\lambda_{1}$ is well-defined with $\lambda_{1}\in[0,\infty]$. It is clear that for $\lambda\in I$,
\[
(\ddbar u_{0})^{m}\wedge\omega^{n-m} \leq (1-\lambda u_{0})^{m}f^m\omega^{n}.
\]
This means $u_{0}$ is a supersolution of \eqref{continuity path}. To show $\lambda_{1}>0$, we observe that for $\lambda\in(0,2^{-1}\|u_{0}\|_{L^{\infty}(\Omega)}^{-1})$,
\[
(\ddbar (2u_0))^m\wedge\omega^{n-m}=2^mf^m\omega^n\geq (1-\lambda (2u_0))^mf^m\omega^n,
\]
which implies that $2u_{0}$ is a subsolution of \eqref{continuity path}. By Theorem \ref{existence under sub-super} (with $\un{u}=2u_{0}$ and $\ov{u}=u_{0}$), we have $(0,2^{-1}\|u_{0}\|_{L^{\infty}(\Omega)}^{-1})\subseteq I$, so $\lambda_{1}\geq 2^{-1}\|u_{0}\|_{L^{\infty}(\Omega)}^{-1}>0$.

To show $\lambda_{1}<\infty$, for any $\lambda\in I$, Maclaurin's inequality shows
\[
\Delta_{\omega}u_{\lambda} \geq n\left(\frac{(\ddbar u_{\lambda})^{m}\wedge\omega^{n-m}}{\omega^{n}}\right)^{\frac{1}{m}}
= n(1-\lambda u_{\lambda})f \geq -(n\lambda\inf_{\Omega}f)u_{\lambda}.
\]
By Theorem \ref{first eigenvalue Laplacian}, we obtain $n\lambda\inf_{\Omega}f \leq \mu_{1}(\Delta_{\omega})$ and so
\[
\lambda_{1}\leq\frac{\mu_{1}(\Delta_{\omega})}{n\inf_{\Omega}f}<\infty.
\]

\bigskip
\noindent
{\bf Step 2.} $\|u_{\lambda}\|_{L^{\infty}(\Omega)}\rightarrow\infty$ as $\lambda\rightarrow\lambda_{1}$.
\bigskip

Indeed, if not, there exists a constant $M$ and a sequence $\lambda_j\in(0,\lambda_1)$, converging to $\lambda_1$ such that
\[
\|u_{\lambda_j}\|_{L^{\infty}(\Omega)} \leq M.
\]
For any $\lambda \in I$, it is clear that
\[
(\ddbar u_{\lambda})^{m}\wedge\omega^{n-m} \geq f^{m}\omega^{n} = (\ddbar u_{0})^{m}\wedge\omega^{n-m}
\]
and
\[
(\ddbar u_{\lambda})^{m}\wedge\omega^{n-m} \leq (1+\lambda_{1}M)^{m}f^{m}\omega^{n}
= (\ddbar(1+\lambda_{1}M)u_{0})^{m}\wedge\omega^{n-m}.
\]
Then the domination principle shows that
\[
(1+\lambda_{1}M)u_{0} \leq u_{\lambda} \leq u_{0},
\]
so by Theorem \ref{a priori estimates} (with $\un{u}=(1+\lambda_{1}M)u_{0}$ and $w=u_{0}$), we obtain
\[
\|u_{\lambda}\|_{L^{\infty}(\Omega)} \leq C, \quad
\sup_{\Omega}|\nabla^{2}u_{\lambda}| \leq C\sup_{\Omega}|\de u_{\lambda}|^{2}+C,
\]
for $C$ independent of $\lambda$. By the blow-up argument in \cite[Section 6]{CP22}, we obtain
\[
\|u_{\lambda}\|_{C^{2}(\Omega)} \leq C.
\]
By standard elliptic theory, we can choose a subsequence $u_{\lambda_j}$ which smoothly converges to $u_* \in \mSH(\Omega)\cap C^\infty(\ov\Omega)$, so that $(u_*,\lambda_1)$ is a solution of (\ref{continuity path}). Choosing
\[
\delta \in \left(0,\|u_*\|_{L^{\infty}(\Omega)}^{-1}\right) \ \text{and} \
C_{\delta} = \frac{1}{1-\delta\|u_*\|_{L^{\infty}(\Omega)}},
\]
we compute
\[
\begin{split}
& (\ddbar(C_{\delta}u_*))^m\wedge\omega^{n-m}
= C_{\delta}^m(1-\lambda_1u_*)^mf^m\omega^n  \\[1mm]
= {} & (C_{\delta}-\lambda_1(C_{\delta}u_*))^mf^m\omega^n
\geq (1-(\lambda_1+\delta)(C_{\delta}u_*))^mf^m\omega^n.
\end{split}
\]
However, by Theorem \ref{existence under sub-super} (with $\un{u}=C_{\delta}u_{0}$ and $\ov{u}=u_{0}$), there exists a solution $u_\lambda$ of (\ref{continuity path}) for $\lambda_1\leq\lambda\leq\lambda_1+\delta$, a contradiction with the definition of $\lambda_1$.

\bigskip
\noindent
{\bf Step 3.} $\lambda_{1}$ is the required first eigenvalue.
\bigskip

We consider the normalized family
\begin{equation*}
v_\lambda:=\frac{u_\lambda}{\|u_\lambda\|_{L^{\infty}(\Omega)}}
\end{equation*}
for $\lambda<\lambda_1$. Then $v_\lambda$ satisfies $\|v_\lambda\|_{L^{\infty}(\Omega)}=1$ and
\begin{equation}\label{Dirichlet v lambda}
\begin{cases}
(\ddbar v_\lambda)^m\wedge\omega^{n-m}
=(\|u_\lambda\|_{L^{\infty}(\Omega)}^{-1}-\lambda v_\lambda)^m f^m\omega^n & \text{in $\Omega$},  \\
v_\lambda=0 & \text{on $\partial\Omega$}.
\end{cases}
\end{equation}
Using Step 2 and $\lambda<\lambda_{1}$, we have
\[
(\ddbar v_\lambda)^m\wedge\omega^{n-m} \leq (1+\lambda_{1})^{m}f^{m}\omega^{n}
= (\ddbar (1+\lambda_{1})u_{0})^m\wedge\omega^{n-m}.
\]
Then the domination principle implies
\begin{equation}\label{lower barrier u lambda}
(1+\lambda_{1})u_{0} \leq v_{\lambda}.
\end{equation}

To apply Theorem \ref{a priori estimates}, we need to construct an appropriate supersolution, $w$. By compactness of $m$-subharmonic functions, there exists a subsequence $\{\lambda_j\}$, with $\lambda_j\rightarrow \lambda_1$ as $j\rightarrow \infty$, such that the $v_{\lambda_j}$ converge in $L^1_{\mathrm{loc}}$ to some $u_1\in \mSH(\Omega)$. In \cite[Section 4: Step 5]{BZ23a}, Badiane-Zeriahi used the {\it a priori} gradient estimate of B\l ocki \cite{Bl09} for the complex Monge-Amp\`ere equation to upgrade this to $C^0$-convergence; the corresponding gradient estimate for the complex $m$-Hessian equation is not currently known. Instead, we will use the stability estimate, Theorem \ref{stability estimate}, to show this convergence. After that, we can follow \cite[Section 4: Step 5]{BZ23a} to construct the required $w$.

Specifically, we will show that $\{v_{\lambda_{j}}\}$ is a Cauchy sequence in $C^{0}(\ov{\Omega})$. Using $u_{0}=0$ on $\de\Omega$, for any $\ve>0$,  there exists $\Omega'\Subset\Omega$ such that
\[
\|u_{0}\|_{L^{\infty}(\Omega\setminus\Omega')} \leq \ve
\]
and so
\[
\|v_{\lambda_{k}}-v_{\lambda_{l}}\|_{L^{\infty}(\Omega\setminus\Omega')}
\leq 2\|Cu_{0}\|_{L^{\infty}(\Omega\setminus\Omega')} \leq 2\ve,
\]
by \eqref{lower barrier u lambda}. On the other hand, we have $v_{\lambda_{j}}\to u_{1}$ in $L_{\mathrm{loc}}^{1}(\Omega)$, so
\[
\|v_{\lambda_{k}}-v_{\lambda_{l}}\|_{L^{1}(\Omega')} \leq \ve.
\]
Using $\|v_{\lambda_{k}}\|_{L^{\infty}(\Omega)}=\|v_{\lambda_{l}}\|_{L^{\infty}(\Omega)}=1$, we have
\[
\|v_{\lambda_{k}}-v_{\lambda_{l}}\|_{L^{p}(\Omega')}
\leq \|v_{\lambda_{k}}-v_{\lambda_{l}}\|_{L^{\infty}(\Omega')}^{\frac{p-1}{p}}\cdot\|v_{\lambda_{k}}-v_{\lambda_{l}}\|_{L^{1}(\Omega')}^{\frac{1}{p}}
\leq 2^{\frac{p-1}{p}} \ve^{\frac{1}{p}}.
\]
Thus
\[
\|v_{\lambda_{k}}-v_{\lambda_{l}}\|_{L^{p}(\Omega)}
\leq \|v_{\lambda_{k}}-v_{\lambda_{l}}\|_{L^{p}(\Omega\setminus\Omega')}+\|v_{\lambda_{k}}-v_{\lambda_{l}}\|_{L^{p}(\Omega')}
\leq 2\ve+2^{\frac{p-1}{p}}\ve^{\frac{1}{p}}.
\]
Using Step 2, for sufficiently large $k$ and $l$.
\[
\|u_{\lambda_{k}}\|_{L^{\infty}(\Omega)}^{-1}+\|u_{\lambda_{l}}\|_{L^{\infty}(\Omega)}^{-1}\leq\ve.
\]
Now Theorem \ref{stability estimate} shows that:
\[
\begin{split}
& \|v_{\lambda_k}-v_{\lambda_l}\|_{L^{\infty}(\Omega)} \\
\leq {} & \|f\|_{L^{\infty}(\Omega)}\cdot\left\|(\|u_{\lambda_{k}}\|_{L^{\infty}(\Omega)}^{-1}-\lambda_{k} v_{\lambda_{k}})^m
-(\|u_{\lambda_{l}}\|_{L^{\infty}(\Omega)}^{-1}-\lambda_{l} v_{\lambda_{l}})^m \right\|_{L^{p}(\Omega)}^{\frac{1}{m}} \\
\leq {} & C\left\|(\|u_{\lambda_{k}}\|_{L^{\infty}(\Omega)}^{-1}-\lambda_{k} v_{\lambda_{k}})
-(\|u_{\lambda_{l}}\|_{L^{\infty}(\Omega)}^{-1}-\lambda_{l} v_{\lambda_{l}}) \right\|_{L^{p}(\Omega)}^{\frac{1}{m}} \\
\leq {} & C\left(\|\lambda_k v_{\lambda_k}-\lambda_l v_{\lambda_l}\|_{L^{p}(\Omega)}+C\ve\right)^{\frac{1}{m}},
\end{split}
\]
where we used the elementary equality $a^{m}-b^{m}=(a-b)(a^{m-1}+\ldots+b^{m-1})$ in the third line. So combining the previous two estimates gives:
\[
\|v_{\lambda_{k}}-v_{\lambda_{l}}\|_{L^{\infty}(\Omega)}
\leq C\left(2\ve+2^{\frac{p-1}{p}}\ve^{\frac{1}{p}}
+C\ve\right)^{\frac{1}{m}}.
\]
Then $\{v_{\lambda_{j}}\}$ is a Cauchy sequence in $C^{0}(\ov{\Omega})$. We thus conclude that
\begin{equation}\label{C 0 convergence}
v_{\lambda_{j}} \to u_{1} \ \text{in $C^{0}(\ov\Omega)$}.
\end{equation}
We can now follow the argument of \cite[Section 4: Step 5]{BZ23a} to construct a super-barrier. Since $\|u_{1}\|_{C^{0}(\ov\Omega)}=1$ and $u_{1}=0$ on $\Omega$, then there exists $z_{0}\in\Omega$ such that $u_{1}(z_{0})=-1$. Let $B$ be a ball such that
\[
z_{0} \in B \Subset \Omega, \ \text{and} \
u_{1} \leq -\frac{2}{3} \ \text{in $B$}.
\]
For $j\gg1$, we have $\lambda_{j}\geq\frac{\lambda_{1}}{2}$ and $-v_{\lambda_{j}}\geq\frac{1}{3}$ in $B$. Then
\[
(\ddbar v_{\lambda_{j}})^m\wedge\omega^{n-m}
\geq (-\lambda_{j}v_{\lambda_{j}})^m f^{m}\omega^{n}
\geq \left(\frac{\lambda_{1}f}{6}\right)^{m}\omega^{n} \ \text{in $B$}.
\]
By Maclaurin's inequality,
\[
\Delta_{\omega}v_{\lambda_{j}} \geq n\left(\frac{(\ddbar v_{\lambda_{j}})^{m}\wedge\omega^{n-m}}{\omega^{n}}\right)^{\frac{1}{m}}
= \frac{n\lambda_{1}f}{6} \ \text{in $B$}.
\]
Let $\theta_{B}$ be a smooth function in $\ov{\Omega}$ such that $0\leq\theta_{B}\leq1$, $\mathrm{Supp}(\theta_{B})\subset B$ and $\theta_{B}\equiv1$ near $z_{0}$. Then there exists $w\in\SH_{1}(\Omega)\cap C^{\infty}(\ov{\Omega})$ solving the Dirichlet problem:
\[
\begin{cases}
\Delta_{\omega}w = \frac{n\lambda_{1}}{6}f\theta_{B} & \text{in $\Omega$}, \\
w = 0 & \text{on $\de\Omega$}.
\end{cases}
\]
Using the comparison principle, we obtain
\[
v_{\lambda_{j}} \leq w < 0 \ \text{in $\Omega$}, \ \ w = 0 \ \text{on $\de\Omega$}.
\]
By combining this with \eqref{lower barrier u lambda} and arguing as in Step 2, we obtain
\[
\|v_{\lambda_j}\|_{C^{2}(\Omega)} \leq C.
\]
For any $\Omega'\Subset\Omega$, using \eqref{C 0 convergence} and $u_{1}<0$ in $\Omega$, there exists a constant $c_{\Omega'}$ such that
\[
v_{\lambda_j} \leq -c_{\Omega'} < 0 \ \text{in $\Omega'$},
\]
which shows that \eqref{Dirichlet v lambda} is non-degenerate in $\Omega'$. Then by standard elliptic theory, we obtain higher order estimates of $v_{\lambda_{j}}$ in $\Omega'$. Thus we can take a further subsequence of the $v_{\lambda_j}$ which converges to $u_1$ in $C^\infty(\Omega)\cap C^{1,1}(\ov{\Omega})$. We conclude $(\lambda_{1},u_{1})$ solves the eigenvalue problem \eqref{CHE eigenvalue Section}.

\bigskip
\noindent
{\bf Step 4.} Uniqueness of $(\lambda_{1},u_{1})$.
\bigskip

This immediately follows from Theorem \ref{uniqueness in E1} below.
\end{proof}

We now prove a strong uniqueness result. It shows that the eigenfunction constructed in Theorem \ref{Theorem 1 in Section 5} is unique in the finite energy class $\mathcal{E}^1_m(\Omega)$, significantly improving on the results in \cite{BZ23b}.

\begin{theorem}\label{uniqueness in E1}
Suppose that $w\in \mathcal{E}^1_m(\Omega)$ and $\alpha > 0$ are such that $(w, \alpha)$ is a weak solution to \eqref{CHE eigenvalue}. Then $\alpha=\lambda_{1}$ and $w = \theta u_{1}$ for some $\theta\geq 0$.
\end{theorem}

\begin{proof}
For notational convenience, we will replace $f^{m}$ with $\binom{n}{m}^{-1}f^{m}$ and write $u=u_{1}$, $\lambda = \lambda_1$, throughout the proof, so that:
\[
\begin{cases}
\sigma_{m}^{1/m}(u) = -\lambda uf & \text{in $\Omega$}, \\
u = 0 & \text{on $\de\Omega$}, \\
\inf_{\Omega}u = -1,
\end{cases}
\]
and $\sigma_{m}^{1/m}(w) = -\alpha wf$ with $w\in\mathcal{E}^1_m(\Omega)$.

First, we show $\alpha=\lambda$; the argument is essentially due to Le \cite[Proposition 5.6]{Le18}. Since $u, w\in\mathcal{E}^1_m(\Omega)$, we can integrate by parts, so that:
\[
\int_{\Omega}(-w)(\ddbar u)^{m}\wedge\omega^{n-m} =\int_{\Omega}(-u)\ddbar w\wedge(\ddbar u)^{m-1}\wedge\omega^{n-m}.
\]
By the mixed Hessian inequality of Dinew-Lu \cite[Theorem 3.10]{DL15}, we also have that:
\begin{equation}\label{Dinew-Lu inequality}
\binom{n}{m} \ddbar w\wedge(\ddbar u)^{m-1}\wedge\omega^{n-m} \geq \sigma_{m}^{\frac{1}{m}}(w)\,\sigma_{m}^{\frac{m-1}{m}}(u) \, \omega^n.
\end{equation}
We now combine these to compute:
\[
\begin{split}
& \int_{\Omega}(-w)(-\lambda u f)^{m}\omega^{n} \\
= {} & \binom{n}{m}\int_{\Omega}(-w)(\ddbar u)^{m}\wedge\omega^{n-m} \\
= {} & \binom{n}{m}\int_{\Omega}(-u)\ddbar w\wedge(\ddbar u)^{m-1}\wedge\omega^{n-m} \\
\geq {} & \int_{\Omega}(-u)\sigma_{m}^{\frac{1}{m}}(w)\,\sigma_{m}^{\frac{m-1}{m}}(u)\,\omega^{n} \\
= {} & \lambda^{-1}\alpha\int_{\Omega}(-w)(-\lambda uf)^{m}\omega^{n}.
\end{split}
\]
It follows that $\lambda\geq \alpha$; interchanging the roles of $(u, \lambda)$ and $(w, \alpha)$, we obtain $\alpha=\lambda$.

\medskip

We now show $w = u$. Denote the derivatives of $\sigma_{m}$ and $\sigma_{m}^{1/m}$ at $u$ by $\sigma_{m}^{i\ov{j}}$ and $F^{i\ov{j}}$, respectively. We define the linear elliptic operator $L=F^{i\ov{j}}\de_{i}\de_{\ov{j}}$, which is degenerate on $\p\Omega$. We have
\[
L = F^{i\ov{j}}\de_{i}\de_{\ov{j}} = \frac{1}{m}\sigma_{m}^{\frac{1}{m}-1}\sigma_{m}^{i\ov{j}}\,\de_{i}\de_{\ov{j}}
= \frac{1}{m}(-\lambda u)^{1-m}\sigma_{m}^{i\ov{j}}\,\de_{i}\de_{\ov{j}}
\]
and
\begin{equation}\label{uniqueness in E1 Lu}
Lu = F^{i\ov{j}}u_{i\ov{j}} = -\lambda uf.
\end{equation}
Note that, by \eqref{Dinew-Lu inequality} and $\alpha = \lambda$, we also have:
\[
\binom{n}{m}\frac{\ddbar w \wedge (\ddbar u)^{m-1}\wedge\omega^{n-m}}{\omega^{n}} \geq (-\lambda uf)^{m-1}(-\lambda w f),
\]
which implies
\[
Lw\geq-\lambda w f \geq 0.
\]

Define $\theta := \sup\{a \geq 0\ |\ w \leq a u \ \text{in $\Omega$}\}$. Since $w\in\mathcal{E}^1_m(\Omega)$, then $w\leq0$ in $\Omega$ and $\theta$ is well-defined. Suppose for a contradiction that $w\not= \theta u$. By combining \eqref{uniqueness in E1 Lu} with the last inequality, we have:
\begin{equation}\label{w theta u bound}
L(w - \theta u) \geq -\lambda(w - \theta u) f \geq 0.
\end{equation}
Note that the operator $L$ in not uniform elliptic near $\de\Omega$ and so the Hopf lemma can not be applied here directly. We will now show that, by a modification of the Hopf lemma, we can increase $\theta$ slightly, which will give us a contradiction.

We first claim that there exists $c_{0}>0$ such that
\begin{equation}\label{uniqueness in E1 claim}
F^{i\ov{j}}u_{i}u_{\ov{j}} \geq c_{0}(-\lambda uf)|\de u|^{2} \ \text{in $\Omega$}.
\end{equation}
Assuming the claim for now, we show how the result follows.

By $\Delta_{\omega}u\geq0$ and the usual Hopf lemma, $|\de u|=\frac{\de u}{\de\nu}>0$ on $\de\Omega$, where $\nu$ is the outer normal vector to $\de\Omega$. Then there exist $c_{1},\delta_{0}>0$ such that
\[
|\de u|^{2} > c_{1} \ \ \text{in $\{-\delta_{0}<\rho<0\}$}.
\]
Together with the claim \eqref{uniqueness in E1 claim}, this implies
\begin{equation}\label{uniqueness in E1 de u}
F^{i\ov{j}}u_{i}u_{\ov{j}}  \geq c_{0}c_{1}(-\lambda uf) \ \ \text{in $\{-\delta_{0}<\rho<0\}$}.
\end{equation}
Define $v = e^{-Au} - 1 - u$. Using \eqref{uniqueness in E1 Lu} and \eqref{uniqueness in E1 de u}, in $\{-\delta_{0}<\rho<0\}$, we compute
\[
\begin{split}
Lv = {} & e^{-Au}\left(A^{2}F^{i\ov{j}}u_{i}u_{\ov{j}}-AF^{i\ov{j}}u_{i\ov{j}}\right) - F^{i\ov{j}}u_{i\ov{j}} \\
\geq {} & e^{-Au}(-\lambda uf)\left(c_{0}c_{1}A^{2}-A-e^{Au}\right).
\end{split}
\]
Choosing $A\gg1$, we obtain
\begin{equation}\label{uniqueness in E1 Lv}
Lv \geq 0 \ \ \text{in $\{-\delta_{0}<\rho<0\}$}.
\end{equation}

By \eqref{w theta u bound} and using the strong maximum principle for $L$ on the set ${\{\rho < -\tfrac{1}{2}\delta_0\}}$ (where $L$ is uniformly elliptic), we have that $\max_{\{\rho \leq -\delta_0\}}(w - \theta u) < 0$. It follows that there exists some $\ve_{0}\ll1$ such that
\[
w - \theta u +\ve_{0}v \leq 0 \ \ \text{on $\{\rho=-\delta_{0}\}$}.
\]
Combining this with \eqref{w theta u bound} and \eqref{uniqueness in E1 Lv},
\[
\begin{cases}
L(w - \theta u +\ve_{0}v) \geq 0 & \text{in $\{-\delta_{0}<\rho<0\}$}, \\[1mm]
w - \theta u +\ve_{0}v \leq 0 & \text{on $\{\rho=-\delta_{0}\}$}, \\[1mm]
\displaystyle{\limsup_{x\to z}}\,(w - \theta u +\ve_{0}v)(x) \leq 0 & \text{for all $z\in\de\Omega = \{\rho=0\}$}. \\
\end{cases}
\]
Using the maximum principle, we see that
\[
w - \theta u \leq -\ve_{0}v \leq \ve_{0}u \ \ \text{in $\{-\delta_{0}<\rho<0\}$}.
\]
Since $\max_{\{\rho \leq -\delta_0\}} (w - \theta u) < 0$, then
\begin{equation}\label{uniqueness in E1 eqn 2}
w \leq (\theta + \e) u\text{ in }\Omega
\end{equation}
for some $0 < \e \ll \e_0$, contradicting the definition of $\theta$.

\smallskip

Now it suffices to prove claim \eqref{uniqueness in E1 claim}. For any point $x_{0}\in\Omega$, we choose a holomorphic normal coordinate system $(U,\{z^{i}\}_{i=1}^{n})$ for $g$ centered at $x_{0}$ such that
\[
u_{i\ov{j}} = \delta_{ij}u_{i\ov{i}}, \ \ u_{1\ov{1}} \geq u_{2\ov{2}} \geq \cdots \geq u_{n\ov{n}}, \ \ \text{at $x_{0}$}.
\]
Let $\Lambda_{1}\geq\Lambda_{2}\geq\ldots\geq\Lambda_{n}$ be the eigenvalues of $\ddbar u$ with respect to $\omega$. It then follows that
\[
\sigma_{m}^{n\ov{n}} \geq \sigma_{m}^{n-1\ov{n-1}} \geq \cdots \geq \sigma_{m}^{1\ov{1}}.
\]
By \cite[Lemma 3.1]{CW01} (see also \cite[Lemma 2.2]{HMW10}), we have
\[
\sigma_{m}^{1\ov{1}} \geq \frac{m}{n\Lambda_{1}}\sigma_{m}(u) = \frac{m}{n\Lambda_{1}}(-\lambda uf)^{m}.
\]
In Step 3 of the proof of Theorem \ref{Theorem 1 in Section 5}, we know that $u\in C^{1,1}(\ov{\Omega})$, which implies $\Lambda_{1}\leq C_{0}$ for some constant $C_{0}$. Hence,
\[
F^{i\ov{j}}u_{i}u_{\ov{j}}
\geq \frac{1}{m}(-\lambda uf)^{1-m}\sigma_{m}^{1\ov{1}}|\de u|^{2} \geq c_{0}(-\lambda uf)|\de u|^{2},
\]
as claimed.
\end{proof}

Finally, we prove the Rayleigh quotient formula for $\lambda_1$ (Theorem \ref{Theorem 2}).

\begin{theorem}[Theorem \ref{Theorem 2}]\label{variational characterization}
Suppose that $(\Omega, \omega)$ and $f$ are as in Theorem \ref{Theorem 1}. Then the eigenvalue can be characterized as:
\[
\lambda_{1}(\Omega, f)^{m} = \min\left\{\frac{E_{m}(u)}{I_{m}(u)}~\middle|~u\in\mathcal{E}_{m}^{1}(\Omega),\ u\neq0\right\}.
\]

\end{theorem}
\begin{proof}
We will follow broadly the proof of Badiane-Zeriahi in \cite{BZ23b}. Define:
\begin{equation}\label{alpha def}
\begin{split}
\alpha^m := {} & \inf\left\{\frac{E_{m}(u)}{I_{m}(u)}~|~u\in\mathcal{E}_{m}^{1}(\Omega),\ u\neq0\right\} \\
= {} &\inf\left\{E_m(u)~|~u\in\mathcal{E}_m^1(\Omega), I_m(u)=1, u\neq0 \right\}.
\end{split}
\end{equation}
We break the proof into four steps.

\bigskip
\noindent
{\bf Step 1.} $\alpha>0$.
\bigskip

Suppose $u\in \mathcal{E}^1_m(\Omega)$. By Corollary \ref{CP Dirichlet corollary}, let $u_{0}\in\mSH({\Omega})\cap C^{\infty}(\ov\Omega)$ be the solution to:
\[
\begin{cases}
(\ddbar u_{0})^{m}\wedge\omega^{n-m} =  f^m\omega^{n} & \text{in $\Omega$}, \\
u_{0} = 0 & \text{on $\de\Omega$},
\end{cases}
\]
and set $A = (m+1)!\norm{u_{0}}_{L^\infty(\Omega)}^m$. Then by Proposition \ref{Poincare type inequality proposition}, we have:
\[
I(u) = \int_\Omega (-u)^{m+1}\Hm(u_{0}) \leq A (\sup_{\Omega} f)^m E_m(u);
\]
since $u$ is arbitrary, we conclude $\alpha \geq (A\sup_\Omega f^m)^{-1} > 0$.

\bigskip
\noindent
{\bf Step 2.} The infimum in \eqref{alpha def} is attained.
\bigskip

Consider some sequence $w_j\in \mathcal{E}^1_m(\Omega)$ with $I(w_j) = 1$ and $E_m(w_j)\searrow \alpha^m$. Then by Proposition \ref{weak compactness proposition}, after possibly taking a subsequence, the $w_j$ converge in $L^1_\mathrm{loc}$ to some $w\in \mathcal{E}^1_m(\Omega)$. Since $E_m$ is lower semi-continuous, we conclude $\alpha^m = E_m(w)$, by Fatou's lemma and the definition of $\alpha$.

\bigskip
\noindent
{\bf Step 3.} The pair $(w,\alpha)$ is a weak solution to \eqref{CHE eigenvalue}.
\bigskip

By definition, the function $w$ is a minimizer of the functional $\Phi: \mathcal{E}^1_m(\Omega)\rightarrow \R$, defined by:
\[
\Phi(u) := E_m(u) - \alpha^mI_m(u), \quad u\in\mathcal{E}^1_m(\Omega).
\]
As discussed at the end of the Appendix, it follows that $w$ is a weak solution to the eigenvalue equation (with $\alpha$ instead of $\lambda_1$ of course).

\bigskip
\noindent
{\bf Step 4.} $\alpha=\lambda_{1}$.
\bigskip

This follows immediately from Theorem \ref{uniqueness in E1}.
\end{proof}

\section{Applications}
\label{Applications}

\subsection{Bifurcation} We now present several applications of our main theorems. The first is a standard bifurcation theorem, originally found in Lions \cite[Corollary 2]{Lions85}, which is similar to ones for the first eigenvalue of linear elliptic operators. Compared to \cite[Theorem 1.2]{BZ23b}, we have no restriction on $m$, and are able to additionally prove uniqueness of the solution, although of course our assumptions on the right-hand side are more restrictive.

\begin{theorem}[Theorem \ref{Theorem 3}]\label{bifurcation}
Let $\psi(z,s)$ be a smooth (strictly) positive function on $\ov{\Omega}\times(-\infty,0]$ such that $\p_s \psi \geq -\gamma_0 > -\lambda_1$, where $\lambda_1 = \lambda_1(\Omega, 1)$ is the first eigenvalue of $\Hm$ associated to $\omega^n$.  Then the equation:
\begin{equation}\label{application equation}
\begin{cases}
\sigma_m(u) = \psi(z, u)^m &\text{ in }\Omega,\\
u = 0 &\text{ on }\p\Omega
\end{cases}
\end{equation}
admits a unique solution $u\in \mSH(\Omega)\cap C^\infty(\ov{\Omega})$.
\end{theorem}

\begin{proof}
We first show existence of a solution. It is well-known principle that this should follow from the existence of a subsolution (see e.g. \cite{CKNS85,Guan98}), but we were unable to locate an exact reference for our setting. We thus include the proof here.

We first construct a subsolution. Fix $\gamma \in (\gamma_0, \lambda_1)$ and let $u_\gamma$ be the corresponding solution to \eqref{continuity path} with $f = \binom{n}{m}^{-1}$. Then for any constant $C \geq \norm{\psi(z, 0)}_{L^\infty}$, we have:
\[
\sigma_m^{1/m}(Cu_\gamma) = C (1-\gamma u_\gamma) \geq \norm{\psi(z, 0)}_{L^\infty} - C\gamma_0 u_\gamma \geq \psi(z, Cu_\gamma);
\]
see \cite[Proposition 4.2]{BZ23a}. We conclude that $\ul{u} := Cu_\gamma$ is a subsolution to \eqref{application equation}.

We now construct our solution. Consider the continuity path:
\begin{equation}\label{bifurcation continuity path}
\begin{cases}
\sigma_m^{1/m}(u_t) = t\psi(z, u_t)+(1-t)\sigma_m^{1/m}(\underline{u}) &\text{ in }\Omega, \\
u_t = 0 &\text{ on }\p\Omega.
\end{cases}
\end{equation}
Define
\[
I = \{t\in[0,1]~|~\text{\eqref{bifurcation continuity path} admits a smooth solution $u_{t}$} \}.
\]
We will show that $I$ is non-empty, open, and closed. It is clearly non-empty, since $u_{0}=\underline{u}$ is a solution when $t = 0$.

We show openness. Fix $t\in I$ and denote the linearized operator of \eqref{bifurcation continuity path} by $L_{t}$. Write $F = (F^{i\ov{j}})$ for the first derivative of $\sigma_{m}^{1/m}$. Then
\[
L_{t} = F^{i\ov{j}}(u_{t})\de_{i}\de_{\ov{j}} - t\de_{s}\psi(z,u_{t}).
\]
Let $\mu_{1}(t)$ be the first eigenvalue of the operator $F^{i\ov{j}}(u_{t})\de_{i}\de_{\ov{j}}$. We claim that:
\[
t\psi_{s}(z,u_{t}) > -t\gamma \geq -\mu_1(t).
\]
By concavity of $\sigma_m^{1/m}$, we have:
\[
F^{i\ov{j}}(u_{t})\de_{i}\de_{\ov{j}}(\underline{u}-u_{t}) \geq \sigma_{m}^{1/m}(\underline{u})-\sigma_{m}^{1/m}(u_{t}).
\]
Combining this with $F^{i\ov{j}}(u_{t})\de_{i}\de_{\ov{j}}u_{t}=\sigma_{m}^{1/m}(u_{t})$ and the definition of $\ul{u}$, we have
\[
F^{i\ov{j}}(u_{t})\underline{u}_{i\ov{j}} \geq \sigma_m^{1/m}(\ul{u}) = C - \gamma \ul{u} \geq -\gamma\ul{u}.
\]
Using Theorem \ref{first eigenvalue Laplacian} we obtain $\gamma \leq \mu_1(t)$.

Lemma \ref{max prin c} now shows that $L_{t}$ is injective. By combining this with \cite[Theorem 6.15]{GT01},
$L_{t}$ is also surjective, so that $I$ is open.

To show $I$ is closed, it will suffice to establish uniform estimates on $u_t$. Using \eqref{bifurcation continuity path}, we have
\[
\sigma_m^{1/m}(u_t)-\sigma_m^{1/m}(\underline{u}) = t\left(\psi(z,u_{t})-\sigma_m^{1/m}(\underline{u})\right)
\leq t\big(\psi(z,u_t)-\psi(z,\underline{u})\big).
\]
It follows that
\[
\alpha(t)^{i\ov{j}}\de_{i}\de_{\ov{j}}(u_t-\underline{u})
\leq c(t)(u_t-\underline{u}),
\]
where
\[
\alpha(t)^{i\ov{j}} = \int_{0}^{1}F^{i\ov{j}}(ru_t+(1-r)\underline{u})dr, \ \
c(t) = t\int_{0}^1 \p_s\psi(z, ru_t+(1-r)\underline{u}) dr.
\]
Let $\tilde{\mu}_{1}(t)$ be the first eigenvalue of $\alpha(t)^{i\ov{j}}\de_{i}\de_{\ov{j}}$. By the same argument as above, we know that
\[
c(t) > -\tilde{\mu}_{1}(t).
\]
By Lemma \ref{max prin c}, we obtain $u_t\geq\underline{u}$.

For the $C^2$-estimate, note that $\psi(z,s) \geq \tau > 0$ on $\ov{\Omega}\times[-\|\underline{u}\|_{L^{\infty}(\Omega)}, 0]$ for some small constant $\tau$. It follows that
\[
\sigma_m^{1/m}(u_t) = t\psi(z, u_t)+(1-t)\sigma_m^{1/m}(\underline{u}) \geq  t\psi(z, u_t)+(1-t)\psi(z, \underline{u}) \geq \tau > 0.
\]
If we now choose $\tau_{0}>0$ sufficiently small such that $w=\tau_{0}\rho$ satisfies
\[
\sigma_m^{1/m}(w) = \tau_{0}\sigma_m^{1/m}(\rho) \leq \tau \leq \sigma_m^{1/m}(u_t),
\]
then we may applying Theorem 3.1 and the blow-up argument of \cite[Section 6]{CP22},
to obtain a uniform $C^{2}$ estimate. The higher order estimates follow from the standard Evans-Krylov theory and bootstrapping.

We now establish uniqueness. Suppose that $u$ and $v$ are both solutions. Then
\[
\sigma_m^{1/m}(u)-\sigma_m^{1/m}(v) = \psi(z,u)-\psi(z,v)
\]
and so
\[
\left(\int_{0}^{1}F^{i\ov{j}}(ru+(1-r)v)dr\right)\de_{i}\de_{\ov{j}}(u-v)
= \left(\int_{0}^1 \p_s\psi(z, ru+(1-r)v) dr\right)(u-v).
\]
By applying Lemma \ref{max prin c} in the same way as before, we obtain $u=v$.
\end{proof}

\subsection{Monotonicity}

We now discuss monotonicity property of the first eigenvalue.
\begin{theorem}\label{monotonicity}

Suppose we are in the setting of Theorem \ref{Theorem 1}, and $\Omega'\subset\Omega$ is another strongly $m$-pseudoconvex manifold. Then
\[
\lambda_{1}(\Omega, f) < \lambda_{1}(\Omega', f).
\]

\end{theorem}
\begin{proof}

By Theorem \ref{Theorem 2}, there exists some $w'\in\mathcal{E}^1_m(\Omega')$ such that:
\[
\lambda_1^m(\Omega', f) = \frac{E_{m,\Omega'}(w')}{I_{m,\Omega'}(w')}.
\]
By Theorem \ref{Subextension Theorem}, we can find some $w\in \mathcal{E}^1_m(\Omega)$ with $w \leq w'$ and $E_{m,\Omega}(w)\leq E_{m,\Omega'}(w')$. We also clearly have:
\[
I_{m, \Omega'}(w') \leq I_{m, \Omega}(w);
\]
we claim this inequality is actually strict. If not, then $w \equiv 0$ on $\Omega\setminus \Omega'$ (since $\Omega\setminus\Omega'$ has positive measure), and $w = w'$ on $\Omega'$. But then $w$ would attain its maximum at an interior point, and so by the strong maximum principle, $w\equiv0$, which is a contradiction.

We conclude then that:
\[
\lambda_1^m(\Omega, f) \leq \frac{E_{m,\Omega}(w)}{I_{m,\Omega}(w)} < \lambda_1^m(\Omega', f),
\]
proving the theorem.
\end{proof}

\subsection{Geometric Bounds}

Motivated in part by Theorem \ref{bifurcation}, we conclude with some geometric bounds on $\lambda_1$; our results are inspired by \cite{Le18}. The variational characterization in Theorem \ref{Theorem 2} gives an upper bound on manifolds with non-negative Ricci curvature:
\begin{theorem}[Theorem \ref{upper bound manifold}]
Suppose that the Ricci curvature $\Ric(\omega)$ is non-negative. Let $R > 0$ be the largest number such that there exists a geodesic ball $B_R(p)\subseteq \Omega$ with $r^{2}(z) := \dist_\omega^{2}(z, p)$ smooth and plurisubharmonic on $B_R(p)$. Then:
\[
\lambda_1(\Omega, f)  \leq c(n,m) (\inf_\Omega f)^{-\frac{2m+1}{m+1}}  R^{-\frac{2n}{m+1} - 2} \diam(\Omega)^{\frac{2n}{m+1}}\vol(\Omega)^{-\frac{1}{m+1}}
\|f\|_{L^{m}(\Omega)}^{\frac{m}{m+1}}.
\]
\end{theorem}

\begin{proof}
Let $v$ be the smooth solution to:
\begin{align*}
\begin{cases}
	\sigma_{m}(v)= (4n)^mf^m\omega^n & \text{in $\Omega$}, \\
	v = 0 & \text{on $\de\Omega$}.
\end{cases}
\end{align*}
Consider the non-zero function
\begin{align*}
	w(z)=(\inf_{\Omega} f)(r^2-R^2),
\end{align*}
which is a smooth plurisubharmonic function on the ball $B_R(p)$. By Maclaurin's inequality, we have
	\begin{align*}
		\sigma_m(w)\leq (\Delta_\omega w)^m=(\inf_\Omega f)^m(\Delta_\omega r^2)^m.
	\end{align*}
	Since $\Ric(\omega)\geq 0$, by the Laplacian comparison theorem, one has
\[
\sigma_m(w) \leq (\inf_\Omega f)^m(\Delta_\omega r^2)^m
\leq (\inf_\Omega f)^m (4n)^m \leq \sigma_m(v).
\]
Since $v\leq 0$ on $\partial B_R(p)$, by the comparison principle, we see that $v\leq w$ on $B_R(p)$.

Now, by the variational characterization of $\lambda_1$, we have:
\begin{equation}\label{asfdsa}
\lambda_1^m(\Omega, f)\leq \frac{\int_\Omega(-v)\Hm(v)}{\int_\Omega (-v)^{m+1}f^m\omega^n}.
\end{equation}
By the H\"older inequality, we have
\begin{align*}
\int_{\Omega} (-v)\Hm(v) &= \int_{\Omega} (-v) f^m\omega^n\\
& \leq \left(\int_{\Omega} (-v)^{m+1} f^m\omega^n\right)^{\frac{1}{m+1}}\left(\int_\Omega f^m\omega^n\right)^{\frac{m}{m+1}},
\end{align*}
so by combining with \eqref{asfdsa}, we have:
\[
\lambda_1(\Omega, f)\leq \left(\frac{\int_\Omega f^m\omega^n}{\int_\Omega (-v)^{m+1}f^m\omega^n}\right)^{\frac{1}{m+1}}.
\]
Let $0 < c < 1$. We estimate the denominator by:
\begin{align*}
\int_{\Omega}(-v)^{m+1}f^m\omega^n &\geq \int_{B_R}(-v)^{m+1}f^m\omega^n\\
&\geq \int_{B_{cR}}(-w)^{m+1}f^m\omega^n\\
&\geq (1-c^2)^{m+1}(\inf_\Omega f)^{m+1} R^{2(m+1)}\int_{B_{cR}(p)} f^m\omega^n\\
&\geq (1-c^2)^{m+1}(\inf_\Omega f)^{2m+1} R^{2(m+1)}\vol(B_{cR}(p))  .
\end{align*}

Set $R_1 = \diam(\Omega)$, so that $B_{R_1}(p) = \Omega$. By the relative volume comparison, we have
\begin{equation*}
	\frac{\vol(B_{R_1}(p))}{\vol(B_{cR}(p))}\leq \frac{R_1^{2n}}{(cR)^{2n}},
\end{equation*}
which implies
\begin{align*}
	\vol(B_{cR}(p))\geq\frac{c^{2n}R^{2n}}{\diam(\Omega)^{2n}}\vol(\Omega).
\end{align*}
In summary, we have the estimate
\begin{align*}
	\int_{\Omega}(-v)^{m+1}f^m\omega^n
	\geq c^{2n}(1-c^2)^{m+1}(\inf_\Omega f)^{2m+1} R^{2(m+1)+2n}\diam(\Omega)^{-2n}\vol(\Omega).
\end{align*}
We can now choose $c$ to maximize the constant, giving:
\[
\lambda_1(\Omega, f)
\leq c(n,m) (\inf_\Omega f)^{-\frac{2m+1}{m+1}}  R^{-\frac{2n}{m+1} - 2}
\diam(\Omega)^{\frac{2n}{m+1}}\vol(\Omega)^{-\frac{1}{m+1}}
\|f\|_{L^{m}(\Omega)}^{\frac{m}{m+1}}.
\]
\end{proof}

When $\Omega\subset\C^n$, we can simplify the above by computing the volume of $B_{cR}$ explicitly, giving the bound:
\[
\lambda_1(\Omega, f) \leq c(n,m) (\inf_\Omega f)^{-\frac{2m+1}{m+1}} R^{-\frac{2n}{m+1} - 2} \norm{f}_{L^{m}(\Omega)}^{\frac{m}{m+1}}.
\]
Further, if $f = 1$, we can apply Theorem \ref{monotonicity} to the ball $B_R$ to give an upper bound for $\lambda_1(\Omega)$ solely in terms of the in-radius:
\[
\lambda_1(\Omega) \leq c(n,m) R^{-2}.
\]

\medskip

For domains in $\C^n$, we also have a lower bound for the first eigenvalue. For the rest of this subsection, we take $\omega$ to be the K\"ahler form of the Euclidean metric on $\C^n$, and write $\abs{E}$ for the Lebesgue measure of any measurable $E\subseteq \Omega$.

Recall the following version of Aleksandrov's maximum principle \cite[Lemma 9.2]{GT01} (see also \cite{Bl05b, Sze18}):
\begin{lemma} \label{ABP max}
For $v\in C^2(\Omega)\cap C^0(\ov\Omega)$ with $v|_{\p\Omega} = 0$, we have
\begin{equation*}
	\norm{v}_{L^\infty(\Omega)}\leq \omega_{2n}^{-1/2n}\diam(\Omega)\left(\int_{\Gamma^+}\det D^2v \right)^{1/2n},
\end{equation*}
where $\Gamma^+:=\{x\in\Omega~|~ v(x) + Dv(x)\cdot(y-x) \leq v(y)\text{ for all }y\in\Omega\}$, and $\omega_{2n}$ is the volume of the unit ball in $\C^n$.
\end{lemma}

\begin{proposition} \label{lower bound of lambda_1}
We have the following lower bound
\[
\lambda_1(\Omega, f) \geq \frac{1}{2}\omega_{2n}^{1/2n} \diam(\Omega)^{-1} \norm{f}_{L^{2n}(\Omega)}^{-1}.
\]
\end{proposition}
\begin{proof}

Let $(\lambda_1(\Omega, f), u)$ be the unique eigenvalue pair on $\Omega$ (i.e. the unique solution to \eqref{CHE eigenvalue}), with $\inf_\Omega u=-1$.

Applying Lemma \ref{ABP max} to $-u$ on $\Omega$ gives:
\begin{equation}\label{next}
\begin{split}
1 \leq {} & \omega_{2n}^{-1/2n}\diam(\Omega)\left(\int_{\Gamma^+}\det D^2u\right)^{1/2n}  \\
\leq {} & 2\omega_{2n}^{-1/2n}\diam(\Omega)\left(\int_{\Gamma^+}(\det u_{i\ov{j}})^2\right)^{1/2n},
\end{split}
\end{equation}
where the second inequality follows as in \cite{Bl05b}. By Maclaurin's inequality, we have:
\[
\det u_{i\ov{j}} = \frac{(\ddbar u)^n}{\omega^n} \leq \left(\frac{\Hm(u)}{\omega^{n}}\right)^{n/m}.
\]
Using the eigenvalue equation and the $L^\infty$ bound on $u$, \eqref{next} becomes:
\[
1 \leq 2 \omega_{2n}^{-1/2n}\diam(\Omega) \left(\int_{\Gamma^+} (\lambda_1 f)^{2n}\right)^{1/2n}.
\]
Replacing $\Gamma^+$ by $\Omega$ finishes.

\end{proof}

When $f = 1$, we can also derive a lower bound solely in terms of the diameter, as in \cite[Remark 4.1]{BZ23a}; by Theorem \ref{monotonicity}, we know that:
\[
\lambda_1(\Omega) \geq \lambda_1(B(a, R)),
\]
where $a \in \C^n$ and $R = \frac{\diam(\Omega)}{2}$. Then, by using the continuity path \eqref{continuity path} as in \cite{BZ23a}, we have:
\[
\lambda_1(\Omega) \geq 4 \diam(\Omega)^{-2}.
\]
When $\Omega$ is itself a ball, the above inequality is sharper than Proposition \ref{lower bound of lambda_1} by a factor of 4; of course, when $\Omega$ is very thin, $\abs{\Omega}$ may be much smaller than the diameter, so that Proposition \ref{lower bound of lambda_1} provides a better estimate.

Finally, similar to the real case discussed in \cite{Le18}, we expect that $\lambda_1(\Omega)$ should ultimately scale like $\abs{\Omega}^{-1/n}$, although proving this currently seems to require new ideas. More generally, we (perhaps naively) expect $\lambda_1(\Omega, f)$ to scale like $1$ over the $L^{n}$-norm of $f$.


\appendix

\section{Some Pluripotential Theory}
\label{Some Pluripotential Theory}

We briefly sketch some fundamental results for the pluirpotential theory on strongly $m$-pseudoconvex manifolds. Since we assume our manifolds to be K\"ahler, these results follow from standard techniques.

We say a smooth function $u: \Omega\rightarrow \R\cup\{-\infty\}$ is $\omega$ $m$-subharmonic ($m$-subharmonic or $m$-sh for short), if:
\[
(\ddbar u)^k\wedge\omega^{n-m} \geq 0 \text{ for all }1 \leq k \leq m.
\]
Following B\l ocki \cite{Bl05a}, an upper semi-continuous, $L^1_{\mathrm{loc}}(\Omega)$-function $u: \Omega\rightarrow \R\cup\{-\infty\}$ is defined to be $m$-sh if, for any smooth $m$-sh functions $v_1, \ldots, v_{m-1}$, we have:
\[
\ddbar u \wedge \ddbar v_1  \wedge \ldots \wedge \ddbar v_{m-1}\wedge\omega^{n-m}\geq 0.
\]
By G\aa rding's inequality \cite{Gar59}, the two definitions agree for smooth $u$, see \cite{Bl05a}. We write $\mSH(\Omega)$ for the set of all $m$-sh functions (omitting the dependence on $\omega$).

If $u$ is $m$-sh and $\omega$ is flat, then local convolutions can be used to produce a sequence of smooth $\omega$-sh functions, $u_j$, which decrease to $u$; this is an important technical tool for many results. For general $\omega$, these convolutions may fail to be $m$-sh, and so the proof of smooth approximation becomes more difficult. It can be shown to hold on an strongly $m$-pseudoconvex manifold by copying exactly the ideas of Pli\'s \cite{Plis13} and Lu-Nguyen \cite{LN15}; the key technical tool needed is the ``zero-temperature limit" of Berman \cite{Ber15}, which follows from slight modifications to the theorem of Collins-Picard \cite{CP22}. See also \cite[Theorem 3.18]{GN18} and \cite[Proposition 2.9]{KN23}, where smooth approximation on certain Hermitian manifolds with boundary is shown.

We can now follow the classical ideas of Bedford-Taylor \cite{BT76, BT82} to show that if $u\in\mSH(\Omega)\cap L^\infty(\Omega)$, then the complex Hessian measure of $u$:
\[
\Hm(u) := \ddbar( u \ddbar ( \ldots u\ddbar u\wedge\omega^{n-m})\ldots )
\]
is a well-defined Radon measure. Moreover, if $u_j\in \mSH(\Omega)\cap L^\infty(\Omega)$ is a decreasing sequence with $u_j\searrow u$, then the measures $\Hm(u_j)$ converge weakly to $\Hm(u)$. These results follow from the Chern-Levine-Nirenberg inequalities, which are proved by integration by parts using the exhaustion function $\rho$. See also the recent paper of Ko\l odziej-Nguyen \cite[Section 3]{KN23}, where they prove similar results in the more general setting of Hermitian manifolds with boundary.

Another fundamental tool we use repeatedly is the comparison principle:
\begin{proposition}\label{Comparison Principle}
Suppose that $u, v\in\mSH(\Omega)\cap L^\infty(\Omega)$, with:
\[
\limsup_{\substack{x\rightarrow z \\ x\in \Omega}} v(x)  \leq \limsup_{\substack{x\rightarrow z \\ x\in \Omega}} u(x)\text{ for all }z\in \p\Omega.
\]
Then:
\[
\chi_{\{u < v\}} \Hm(v) \leq \chi_{\{u < v\}} \Hm(u).
\]
\end{proposition}
We omit the proof, and refer the reader to Cegrell \cite{Ceg84, Ceg98} and Lu \cite{Lu15}.

An important and immediate consequence is the so-called domination principle:
\begin{proposition} \label{Domination Principle}
Suppose that $u, v\in\mSH(\Omega)\cap L^\infty(\Omega)$, with:
\[
\limsup_{\substack{x\rightarrow z \\ x\in \Omega}} v(x)  \leq \limsup_{\substack{x\rightarrow z \\ x\in \Omega}} u(x)\text{ for all }z\in \p\Omega,
\]
and $\Hm(u) \leq \Hm(v)$. Then $v\leq u$ in $\Omega$.
\end{proposition}
\begin{proof}
Let $\e > 0$. Since $\rho \leq 0$, we may apply the comparison principle to $u$ and $v + \e\rho$, giving:
\[
\chi_{\{u < v + \e\}} \Hm(v+\e\rho) \leq \chi_{\{u < v + \e\}} \Hm(u) \leq \chi_{\{u < v+\e\}}\Hm(v).
\]
By multilinearity of the complex Hessian operator and strict $m$-subharmonicity of $\rho$, we have $\Hm(v + \e\rho) \geq \Hm(v) + c\omega^n$, for some $c > 0$ depending on $\e$. It follows that $\{u < v + \e\rho\}$ has $\omega^n$-measure zero for all $\e>0$, so that $u \geq v$ a.e.; since $u, v$ are $\omega$-subharmonic, standard theory implies the inequality holds everywhere, see e.g. \cite{Lu15}.
\end{proof}

We also need some results about the finite energy class $\mathcal{E}^1_m(\Omega)$ with Dirichlet boundary conditions. Define $\mathcal{E}^0_m(\Omega)$ to be:
\begin{align*}
\mathcal{E}^0_m(\Omega) := \{u \in\mSH(\Omega)\ |\ &u\in L^\infty(\Omega), \int_\Omega \Hm(u) < \infty,\\
&\text{ and } \limsup_{\substack{x\rightarrow z \\ x\in \Omega}} u(x) = 0 \text{ for all }z\in \p\Omega\}.
\end{align*}
Then, following Lu \cite{Lu15} (see also \cite{BBGZ13, Ceg98}) we define the energy of a function $u\in\mSH(\Omega)$ to be:
\[
E_m(u) := \sup\left\{ \int_{\Omega} (-v)\Hm(v)\ \Big|\ v\in \mSH(\Omega)\cap\mathcal{E}^0_m(\Omega),\ u \leq v\right\}.
\]
If $u$ satisfies the Dirichlet boundary condition $\limsup_{\substack{x\rightarrow z \\ x\in \Omega}} u(x) = 0$ and $E_m(u) < \infty$, then we say $u$ has finite energy; we write $\mathcal{E}^1_m(\Omega)$ for the set of all such functions.

Although elements in $\mathcal{E}^1_m(\Omega)$ may be unbounded, they still have a well-defined complex Hessian measure, i.e. for any $u\in \mathcal{E}^1_m(\Omega)$, there exists a unique Radon measure, denoted $\Hm(u)$, such that, if $u_j\in\mathcal{E}^0_m(\Omega)$ is a sequence decreasing to $u$, then the sequence $\Hm(u_j)$ weakly converges to $\Hm(u)$. For $m$-hyperconvex domains, this result was proven by Lu \cite{Lu15}; it is easy to see that the proof works for strongly $m$-pseudoconvex manifolds as well.

We also have some important properties of the energy functional. First, one may show that $E_m$ is continuous and monotone increasing along any decreasing sequence $\{u_j\}\subset \mathcal{E}^1_m(\Omega)$ -- this can be shown using integration by parts and some classical inequalities, exactly following \cite{Lu15}. Along general $L^1$-convergent sequences in $\mathcal{E}^1_m(\Omega)$, $E_m$ is only lower semi-continuous. From this, one deduces:
\begin{proposition}\label{weak compactness proposition}
Fix $C > 0$. Then the set:
\[
\{u\in \mathcal{E}^1_m(\Omega)\ |\ E_m(u)\leq C\}
\]
is compact in $L^1_{\mathrm{loc}}(\Omega)$.
\end{proposition}
\begin{proof}
By B\l ocki's inequality, Proposition \ref{Poincare type inequality proposition}, the above set is bounded in $L^1(\Omega)$, so any subset admits an $L^1_{\mathrm{loc}}$ limit point $u\in \mSH(\Omega)$. From lower semi-continuity, $E_m(u) \leq C$.
\end{proof}

Using these continuity results, an integration by parts arguments can be used to prove the cocycle formula, which says that if $u, v\in\mathcal{E}^1_m(\Omega)$, then:
\[
E(v) - E(u) = \frac{1}{m+1}\sum_{k=0}^m \int_\Omega (u - v) (\ddbar u)^k\wedge (\ddbar v)^{m-k}\wedge \omega^{n-m}.
\]
The special case $u = 0$ gives the formula $E_m(v) = \frac{1}{m+1}\int_\Omega (-v)\Hm(v)$.

The main motivation for considering $\mathcal{E}^1_m(\Omega)$ is that the complex Hessian measure can be interpreted as the derivative of the energy functional $E_m$, via the projection formula:
\begin{proposition}\label{projection formula}
Suppose that $u\in \mathcal{E}^1_m(\Omega)$ and $v\in C^0_c(\Omega)$. Then:
\[
\frac{d}{dt}\Big|_{t=0} E_m(P_m(u + tv)) = \int_\Omega (-v)\Hm(u),
\]
where we define:
\[
P_m(h) := \sup\{ w\in \mSH(\Omega)\ |\ w \leq h\}
\]
for any upper semi-continuous function $h$ on $\Omega$.
\end{proposition}
\begin{proof}

We only sketch the basic points of the proof. It is sufficient to show that:
\[
\lim_{t\searrow 0} \frac{E_m(P_m(u + t v)) - E_m(u)}{t} = \int_\Omega (-v)\Hm(u).
\]
By the cocycle formula, integration by parts, and Cauchy-Schwarz, we have:
\begin{align*}
E(P_m(u + tv)) - E(u) &\geq \int_\Omega (u - P_m(u + tv))\, \Hm(u) \geq t\int_\Omega (-v)\Hm(u).
\end{align*}
Dividing by $t$ and taking the limit gives one inequality.

The other direction follows from the orthogonality property $\int_\Omega (P(\psi) - \psi) \Hm(P(\psi)) = 0$, whose proof is local and hence standard. Using this property, we have:
\begin{align*}
E(P_m(u + tv)) - E(u) &\leq \int_\Omega (u - P_m(u + tv))\, \Hm(P_m(u + tv))\\
& = t\int_\Omega (-v) \Hm(P_m(u + tv)).
\end{align*}
We conclude by the weak convergence $\Hm(P_m(u + tv)) \rightarrow \Hm(u)$ as $t\rightarrow 0$.
\end{proof}

Recall that we have defined:
\[
I_m(u) := \frac{1}{m+1}\int_\Omega (-u)^{m+1}\omega^n.
\]
It now follows from Proposition \ref{projection formula} that critical points of the functional $E_m\circ P_m - \lambda_1^m I_m$ are weak solutions to the eigenvalue equation \eqref{CHE eigenvalue}. Since $\Hm$ is well-defined for finite energy functions, the critical points will also be solutions in the pluripotential theoretic sense.


\begin{thebibliography}{99}

\bibitem{AC20} \AA hag, P., Czy\.{z}, R. {\em Poincar\'e and Sobolev-Type Inequalities for Complex $m$-Hessian Equations}, Results Math (2020), 75--63.

\bibitem{A94} Andrews, B. {\em Contraction of convex hypersurfaces in Euclidean space}, Calc. Var. Partial Differential Equations {\bf 2} (1994), no. 2, 151--171.

\bibitem{BZ23a} Badiane, P., Zeriahi, A. {\em The Eigenvalue Problem for the complex Monge-Amp\`ere operator}, J. Geom. Anal. {\bf 33} (2023), no. 12, Paper No. 367, 44 pp.

\bibitem{BZ23b} Badiane, P., Zeriahi, A. {\em A variational approach to the eigenvalue problem for complex Hessian operators}, preprint, arXiv: 2306.04437.

\bibitem{Ber15} Berman, R. {\em From Monge-Amp\`ere equations to envelopes and geodesic rays in the zero temperature limit}, Math. Z. {\bf 291} (2019), no. 1--2, 365--394.

\bibitem{BBGZ13} Berman, R., Boucksom, S., Guedj, V., Zeriahi, A. {\em A variational approach to complex Monge-Amp\`re equations}, Publ. Math. Inst. Hautes \'Etudes Sci. {\bf 117} (2013), 179--245.

\bibitem{BT76} Bedford, E., Taylor, B. A. {\em The Dirichlet problem for a complex Monge-Amp\`ere equation}, Invent. Math. {\bf 37} (1976), no. 1, 1--44.

\bibitem{BT82} Bedford, E., Taylor, B. A., {\em A new capacity for plurisubharmonic functions}, Acta Math. {\bf 149} (1982), no. 1--2, 1--40.

\bibitem{Bl93} B\l ocki, Z. {\em Estimates for the complex Monge-Amp\`ere operator}, Bull. Polish Acad. Sci. Math. {\bf 41} (1993), no. 2, 151--157 (1994).

\bibitem{Bl05a} B\l ocki, Z. {\em Weak solutions to the complex Hessian equation}, Ann. Inst. Fourier (Grenoble) {\bf 55} (5), (2005), 1735--1756.

\bibitem{Bl05b} B\l ocki, Z. {\em On uniform estimate in Calabi-Yau theorem}, Sci. China Ser. A, {\bf 48} (2005), 244--247.

\bibitem{Bl09} B\l ocki, Z. {\em A gradient estimate in the Calabi-Yau theorem}, Math. Ann. {\bf 344} (2009), no. 2, 317--327.

\bibitem{CKNS85} Caffarelli, L, Kohn, J. J., Nirenberg, L., Spruck, J. {\em The Dirichlet problem for nonlinear second-order elliptic equations. II. Complex Monge-Amp\`ere, and uniformly elliptic, equations}, Comm. Pure Appl. Math. {\bf 38} (2), 209--252 (1985).

\bibitem{Ceg84} Cegrell, U. {\em On the Dirichlet Problem for the Complex Monge-Amp\`ere Operator}, Math. Z. {\bf 185} (1984), no. 2, 247--251.

\bibitem{Ceg98} Cegrell, U. {\em Pluricomplex energy}, Acta Math. {\bf 180} (1998), no. 2, 187--217.

\bibitem{CKZ11} Cegrell, U., Ko\l odziej, S., Zeriahi, A. {\em Maximal subextensions of plurisubharmonic functions}, Ann. Fac. Sci. Toulouse Math. (6) {\bf 20}, Fascicule Sp\'ecial (2011), 101--122.

\bibitem{Ch75} Cheng, S.-Y. {\em Eigenvalue comparison theorems and its geometric applications}, Math. Z. {\bf 143} (1975), no. 3, 289--297.

\bibitem{CW01} Chou, K.-S., Wang, X.-J. {\em A variational theory of the Hessian equation}, Comm. Pure Appl. Math. {\bf 54} (2001), no. 9, 1029--1064.

\bibitem{CHZ23} Chu, J., Huang, L., Zhang, J. {\em Fully non-linear elliptic equations on compact almost Hermitian manifolds}, Calc. Var. Partial Differential Equations {\bf 62} (2023), no. 3, Paper No. 105, 34 pp.

\bibitem{CM21} Chu, J., McCleerey, N. {\em Fully non-linear degenerate elliptic equations in complex geometry}, J. Funct. Anal. {\bf 281} (2021), no. 9, Paper No. 109176, 45 pp.

\bibitem{CM23} Chu, J., McCleerey, N. {\em Lelong Numbers of $m$-Subharmonic Functions Along Submanifolds}, to appear in J. Inst. Math. Jussieu; arXiv: 2204.01963, 2022

\bibitem{CTW19} Chu, J., Tosatti, V., Weinkove, B. {\em The Monge-Amp\`ere equation for non-integrable almost complex structures}, J. Eur. Math. Soc. (JEMS) {\bf 21} (2019), no. 7, 1949--1984.

\bibitem{CP22} Collins, T. C., Picard, S. {\em The Dirichlet problem for the $k$-Hessian equation on a complex manifold}, Amer. J. Math. {\bf 144} (2022), no. 6, 1641--1680.

\bibitem{DK14} Dinew, S., Ko\l odziej, S. {\em A priori estimates for complex Hessian equations}, Anal. PDE {\bf 7} (2014), no. 1, 227--244.

\bibitem{DL15} Dinew, S., Lu, C.H. {\em Mixed Hessian inequalities and uniqueness in the class $\mathcal{E}(X,\omega,m)$}, Math. Z. {\bf 279} (2015), no. 3-4, 753--766.

\bibitem{EH89} Ecker, K., Huisken, G. {\em Immersed hypersurfaces with constant Weingarten curvature}, Math. Ann. {\bf 283} (1989), no. 2, 329--332.

\bibitem{Ev10} Evans, L. {\em Partial differential equations. Second edition}, Grad. Stud. Math., 19 American Mathematical Society, Providence, RI, 2010. xxii+749 pp.

\bibitem{Gar59} G\aa rding, L. {\em An inequality for hyperbolic polynomials}, J. Math. Mech. {\bf 8} (1959), 957--965.

\bibitem{G96} Gerhardt, C. {\em Closed Weingarten hypersurfaces in Riemannian manifolds}, J. Differential Geom. {\bf 43} (1996), no. 3, 612--641.

\bibitem{GT01} Gilbarg, D., Trudinger, N. S. {\em Elliptic partial differential equations of second order}, Reprint of the 1998 edition, Classics in Mathematics, Springer-Verlag, Berlin, 2001. xiv+517 pp.

\bibitem{GN18} Gu, D., Nguyen, N. C. {\em The Dirichlet problem for a complex Hessian equation on compact Hermitian manifolds with boundary}, Ann. Sc. Norm. Super. Pisa Cl. Sci. (5) {\bf 18} (2018), no. 4, 1189--1248.

\bibitem{Guan98} Guan, B. {\em The Dirichlet problem for complex Monge-Amp\`ere equations and Regularity of the Pluricomplex Green Function}, Comm. Anal. Geom., {\bf 6} (4), 687--703 (1998).

\bibitem{Guan14} Guan, B. {\em Second-order estimates and regularity for fully nonlinear elliptic equations on Riemannian manifolds}, Duke Math. J. {\bf 163} (2014), no. 8, 1491--1524.

\bibitem{HMW10} Hou, Z., Ma, X.-N., Wu, D. {\em A second order estimate for complex Hessian equations on a compact K\"ahler manifold}, Math. Res. Lett. {\bf 17} (2010), no. 3, 547--561.

\bibitem{KN23} Ko\l odziej, S., Ngyuen, N. C. {\em Complex Hessian measures with respect to a background Hermitian form}, preprint, arXiv: 2308.10405.

\bibitem{KR89} Koutev, N. D., Ramadanov, I. P. {\em Valeurs propres radiales de l\'op\'erateur de Monge-Amp\`ere complexe}, Bull. Sci. Math. (2) {\bf 113} (1989), no. 2, 195--212.

\bibitem{KR90} Koutev, N. D., Ramadanov, I. P. {\em An eigenvalue problem for the complex Monge-Amp\`ere operator in pseudoconvex domains}, Ann. Inst. H. Poincar\'e C Anal. Non Lin\'eaire {\bf 7} (1990), no. 5, 493--503.

\bibitem{Le18} Le, N. Q. {\em The eigenvalue problem for the Monge-Amp`ere operator on general bounded convex domains}, Ann. Sc. Norm. Super. Pisa Cl. Sci. (5) {\bf 18} (2018), no. 4, 1519--1559.

\bibitem{LS17} Le, N. Q., Savin, O. {\em Schauder estimates for degenerate Monge-Amp\`ere equations and smoothness of the eigenfunctions}, Invent. Math. {\bf 207} (2017), no. 1, 389--423.

\bibitem{Li16} Li, C. {\em A Poho\v{z}aev identity and critical exponents of some complex Hessian equations}, J. Partial Differ. Equ. {\bf 29} (2016), no. 3, 175--194.

\bibitem{Lions85} Lions, P.-L. {\em Two remarks on Monge-Amp\`ere equations}, Ann. Mat. Pura Appl. (4) {\bf 142} (1985), 263--275.

\bibitem{Lu15} Lu, C. H. {\em A variational approach to complex Hessian equations in $\C^n$}, J. Math. Anal. Appl. {\bf 431} (2015), no. 1, 228--259.

\bibitem{LN15} Lu, C. H., Nguyen, V.-D. {\em Degenerate complex Hessian equations on compact K\"ahler manifolds}, Indiana Univ. Math. J. {\bf 64} (2015), no. 6, 1721--1745.

\bibitem{LD21} Luo, H., Dai, G. {\em Bifurcation, a-priori bound and negative solutions for the complex Hessian equation}, J. Appl. Anal. Comput. {\bf 11} (2021), no. 2, 937--963.

\bibitem{Ngu14} Nguyen, N. C. {\em H\"older continuous solutions to complex Hessian equations}, Potential Anal. {\bf 41} (2014), no. 3, 887--902.

\bibitem{Pham08} Pham, H. H. {\em Pluripolar sets and the subextension in Cegrell's classes}, Complex Var. Elliptic Equ. {\bf 53} (2008), no. 7, 675--684.

\bibitem{Plis13} Pli\'s, S. {\em The smoothing of $m$-subharmonic functions}, preprint, arXiv: 1312.1906.

\bibitem{Savin14} Savin, O. {\em A localization theorem and boundary regularity for a class of degenerate Monge-Amp\`ere equations}, J. Differential Equations {\bf 256} (2014), no. 2, 327--388.

\bibitem{Spr05} Spruck, J. {\em Geometric aspects of the theory of fully nonlinear elliptic equations}, Global theory of minimal surfaces, Amer. Math. Soc., Providence, RI, 2005, 283--309.

\bibitem{Sze18} Sz\'ekelyhidi, G. {\em Fully non-linear elliptic equations on compact Hermitian manifolds}, J. Differential Geom. {\bf 109} (2018), no. 2, 337--378.

\bibitem{STW17} Sz\'ekelyhidi, G., Tosatti, V., Weinkove, B. {\em Gauduchon metrics with prescribed volume form}, Acta Math. {\bf 219} (2017), no. 1, 181--211.

\bibitem{TY23} Tong, F., Yau, S.-T. {\em Generalized Monge-Amp\`ere functionals and related variational problems}, preprint, arXiv: 2306.01636.

\bibitem{Tso90} Tso, K. {\em On a real Monge-Amp\`ere functional}, Invent. Math. {\bf 101} (1990), no. 2, 425--448.

\bibitem{Wang94} Wang, X.-J. {\em A class of fully nonlinear elliptic equations and related functionals}, Indiana Univ. Math. J. {\bf 43} (1994), no. 1, 25--54.

\bibitem{WZ22} Wang, J., Zhou, B. {\em Trace Inequalities, Isocapacitary Inequalities and Regularity of the Complex Hessian Equations}, preprint, arXiv: 2201.02061.

\bibitem{Wei16} Wei, W. {\em Uniqueness theorems for negative radial solutions of $k$-Hessian equations in a ball}, J. Differential Equations {\bf 261} (2016), no. 6, 3756--3771.

\bibitem{Zhou13} Zhou, B. {\em The Sobolev inequality for complex Hessian equations}, Math. Z. {\bf 274} (2013), no. 1--2, 531--549.

\end{thebibliography}
\end{document}